\documentclass[smallextended]{svjour3}
\usepackage{amsmath,amssymb}
\usepackage{cite}
\usepackage{exscale}
\usepackage{color}
\usepackage[weather,alpine,misc,geometry]{ifsym}
\usepackage[applemac]{inputenc}
\def\qed{\hfill $\triangle$}
\def\text#1{\mbox{#1}}

\renewcommand {\theenumi} {\rm(\roman{enumi})}

\newcommand{\x}[1]{}
\numberwithin{theorem}{section}
\numberwithin{lemma}{section}
\numberwithin{definition}{section}
\numberwithin{proposition}{section}
\numberwithin{corollary}{section}
\numberwithin{remark}{section}
\numberwithin{rem}{section}
\numberwithin{example}{section}
\numberwithin{exmp}{section}

\newcommand{\cl}{\mbox{\rm cl}^*}
\newcommand{\ra}{\rangle}
\newcommand{\la}{\langle}

\newcommand{\st}{\stackrel}
\newcommand{\Limsup}{\mathop{{\rm Lim}\,{\rm sup}}}
\newcommand{\ox}{\bar{x}}

\def\gph{\mathop{\rm gph\,}}

\def\dom{\mathop{\rm Dom\,}}

\def\N{{\mathbb{N}}}

\def\R{{\mathbb{R}}}
\def\S{{\mathbb{S}}}
\title{ Metric Regularity of  the Sum  of   Multifunctions and Applications.
\thanks{Research  partially supported by the Australian Research Council, Project DP110102011,  by NAFOSTED under grant number  $101.01-2011.56$,   by   ECOS-SUD  under the Project C10E08, by LIA ``FormathVietnam" and by  R\'egion  Limousin.}}
\author{
Huynh Van Ngai% \thanks{The first author is grateful to the University of Limoges and XLIM  for support and hospitality during his visit in september  2011.} 
\and   Nguyen  Huu Tron
 \and
Michel Th\'era \thanks{ The third author would like to thank Alexander Kruger for valuable comments.}
 }

\institute{
V. N Huynh \at
Department of Mathematics, University of
Quy Nhon, 170 An Duong Vuong, Quy Nhon, Viet Nam.\\
\email{nghiakhiem@yahoo.com}
\and 
H.T Nguyen \at
University of
Quy Nhon and Laboratoire XLIM, UMR-CNRS 6172,
 Universit\'e de   Limoges. \\
 \email{huutronnguyen@yahoo.com}
\and
M.Th\'era (\Letter\,)\ \at
Laboratoire XLIM, UMR-CNRS 6172,
 Universit\'e de   Limoges and Adjunct professor, University of Ballarat.
\email{michel.thera@unilim.fr}\\
 }
\date{Received: date / Accepted: date}
\journalname{J Optim Theory Appl}

\begin{document}

\maketitle

\begin{abstract}
In this work, we use the theory of  error bounds to study  metric regularity  of  the sum of two multifunctions,  as well as some important properties of variational systems.  We use an approach based on the metric regularity of epigraphical multifunctions. Our results  subsume some recent results by Durea and Strugariu.
\end{abstract}
%%%%%%%%%%%%%%%%%%%%%%%%%%%
%%%%%%%%%%%%% Introduction %%%%%%%%%%%%%%%%
\section{Introduction}\label{intro}

In this paper, we are especially interested in metric regularity of the sum of two multifunctions. The starting point of the study is the famous Lyusternik-Graves Theorem \cite{Lusternik, Graves}, which reduces the problem of regularity of a strictly differentiable single-valued mapping between Banach spaces to that of its linear approximation. Historical  comments and modern interpretations and extensions of this theorem can be found in \cite{[20],ioffe01}. Particularly, it was observed in Dmitruk, Milyutin \& Osmolovsky \cite{[20]} that the original Lyusternik's proof in  \cite{Lusternik} is applicable to a much more general setting: the sum of a covering at a rate  mapping and a Lipschitz one with suitable constants  is covering at the rate. Extensions to the case of the sum of a metrically regular set-valued mapping and a single-valued Lipschitz map with suitable constants appear in  \cite{ Arut, Arut1, RefAz, dontchev,[21], Dont-Roc, ioffe01 },  (see the references therein for more details).
\vskip 2mm
 For the parametric case, it is well-known (see for instance,  Dmitruk \& Kruger \cite{[18]},  Arag\'on Artacho, Dontchev, Gaydu,  Geoffroy, and Veliov \cite{FAMM}) that,  
 if we perturb a metrically regular mapping  $F$ by a mapping  $g(\cdot,\cdot)$,  Lipschitz with respect to $x$, uniformly in $p$, with a  sufficiently small Lipschitz constant,   then the perturbed mapping    $F(\cdot)+g(\cdot,p)$ is metrically regular for every $p$ near $\bar{p}$. More generally, Ioffe \cite{RefIo} extended  this result to the case of the sum of a metrically regular multifunction and a Lipschitz one,  and also to the more general case,  when if a multifunction $G$ is sufficiently close to the  given metrically regular multifunction $F$ in the  sense given in \cite{RefIo}, then $G$ is necessarily metrically regular, with suitable constants (see also \cite{SIOPT-NT}). 
\vskip 2mm
When  we perturb a   metrically regular  multifunction   by another set-valued mapping which is  pseudo-Lipschitz,  the perturbed mapping, i.e., the sum set-valued mapping fails in general to be metrically regular,  (we refer to the  example in the next section). However, if  for example  the so-called " sum-stability" property (introduced below) holds, then  the metric regularityf, as well as the  pseudo-Lipschitz  property  of the variational system,  remains. Recently,   Durea \& Strugariu \cite{Dustru}  considered the  sum of two set-valued mappings and obtained a result    very similar to openness of the sum of two set-valued mappings. They also  gave some applications to  generalized variational systems. 
\vskip 2mm
Motivated by  the ideas and results  from \cite{Dustru},  we attack these problems by  using  a different approach and with rather different assumptions.  Indeed, using an  approach based on the theory of   error bounds, we study   metric regularity of a  special multifunction called the  epigraphical multifunction associated to $F$ and $G$.  This intermediate result   allows us to study metric regularity/ linear openness of the sum of two set-valued mappings, as well as   metric regularity of the general variational system,  avoiding the strong assumption of the  closedness of the sum multifunction.
\vskip 2mm
The paper is structured  as follows. Section 2 is devoted to preliminaries where we introduce the problem  of generalized parametric inclusions. We give some illustrations through examples and we present  a  small survey on different notion of regularity. In  Section 3, we recall  some recent results on error bounds of parametrized systems  and give, sometimes with  some   modifications,   characterizations of metric regularity of multifunctions  given in   \cite{SIOPT-NT, NTT}. In  Section 4, in the context of Asplund spaces, we estimate the strong slope of the 
lower semicontinuous envelope  of the distance function to the epigraphical multifunction associated to two given multifunctions $F$ and $G$.
% $\varphi_{\mathcal{E}}((x,k),y),$ 
Then, we  give   sufficient conditions as well as  a  point-based condition for metric regularity of this epigraphical multifunction under a coderivative condition. In the last section,  we study   Robinson metric regularity  and Aubin property of a generalized variational system.
%%%%%% Section 1 %%%%%%%%%%%%%%%%%%%%%%%%%%
\section{Preliminaries}
Generalized equations, i.e., inclusions of the type  
\begin{equation}\label{A}
0\in F(x,p),\end{equation}
involving a multifunction $F:X\times P\rightrightarrows Y$ where $X,Y$ are  metric spaces, and $P$ is  a topological space  considered as the space of parameters, 
have been extensively used for
modeling optimization and complementarity problems, as well as variational inequalities since the pioneering
work of Robinson \cite{Rob3,Rob4}. 
The study of generalized equations constitute the core of  the development of set-valued analysis \cite{AF90} which
  is one of the main corner-stones of   variational
analysis, see, e.g.,  books \cite{BZ05, iusem, Dont-Roc, Borisbook1,  penot2012, Roc-Wet}.
A typical example of  (\ref{A}) is given by a parametrized system of inequalities/equalities.  More precisely,  let us consider % 
the system  $(\mathcal{S}),$  consisting of those  points $x$ for which
$$\begin{array}{ll}
   f_i(x,p)\leq 0, \; i \in\{1,\cdots ,k\}, \\
    f_i(x,p)=0,\; i \in\{k+1,\cdots ,k+d\},
    \end{array}
    $$
     where $x\in\R^m $ is the decision variable,  $p\in\R^n $ a parameter and  for each $i\in\{1, k+d\},$ and the $ f_i's$  are functions from $\R^m\times \R^n$ to $\R$.  Setting  $f(x,p) = (f_1(x,p), \cdots, f_k(x,p), f_{k+1}(x,p), \cdots, f_{k+d}(x,p)),$ and 
$$F(x,p) := f(x,p)- \R^k _{-}\times\{0\}^{d},$$  the system  $(\mathcal{S})$
can be reformulated in the form 
      %an important subcase  
 $(\ref{A}$).  Let us  also note  that  $(\ref{A})$ includes  the important subcase of parametrized \textit{ generalized inclusions}:
   \begin{equation}\label{generalized}
0\in H(x)+f(x,p),
\end{equation} 
where $H: X\rightrightarrows Y$ is a set-valued mapping and $f:X\times P\to Y$ is a mapping.
 
Let us consider the perturbed optimization problem $(\mathcal{P})$ 
$$\min_{x\in C} [g(x)- \langle p,x\rangle ],$$
where $g:\R^n\to \R$ is a  Fr\'echet differentiable
% convex $C^{1}-$
 function, and $p\in\R^n$ is a given parameter. The first order optimality condition of problem $(\mathcal{P})$  is given by 
%$$\langle f(x,p), u-x\rangle\geq 0 \;\; \text{ for all }\; u\in C,$$
%or equivalently by 
\begin{equation}\label{oc} p-\nabla g(x)\in N_C(x). \end{equation}
Here  $N_C$ stands for the normal cone mapping  defined by  $$N_C(x) =\{v \in \R^n\,:\;\langle v,y-x\rangle \leq 0\quad \forall y\in C\}$$ if $x\in C$, and $N_C(x) =\emptyset$  otherwise.
% stands for the normal cone mapping  to $C$ at $x$.
Setting $f(x,p)=\nabla g(x)-p,$  relation $(\ref{oc})$  takes the form 
\begin{equation}\label{gvi}0\in f(x,p) + N_C(x).\end{equation}
Hence,  the first order optimality condition  satisfies the generalized variational inequality  $(\ref{gvi})$  and appears as a special case of equation $ (\ref{generalized})$. 

The study of variational properties and   stability of the solutions of  equation (\ref{A}) has attracted a large  interest from  a large number of authors, and we refer the reader to the monographs   \cite{Borisbook1,  Roc-Wet,  Dont-Roc} and the references therein.

Let us first  provide definitions and properties of some
essential notions from set-valued analysis that will be used
throughout this paper.  In what follows,  $X$, $Y$,  etc.,  unless specified otherwise,  are metric spaces,  and we use the same symbol $d(\cdot,\cdot)$
to denote  the distance in all of them  or between a point $x$ to a subset $\Omega$ of one of them :  $d(x,\Omega): = \inf_{u\in \Omega} d(x,u).$
By $B(x,\rho)$ and $\bar{B}(x,\rho)$  we  denote the open and closed  balls
of radius $\rho$ around $x$,  while, if  $X$ is a normed linear space,  we use the notations $B_{X}, \bar{B}_{X}$ for the open and the closed unit balls, respectively.
By a multifunction  (set-valued mapping)  $S :X \rightrightarrows Y$, we mean a mapping from $X$ into the subsets  (possibly empty) of $Y$. We denote by  $ \gph S $ the graph of $S$, that is the set $\{(x,y)\in X\times Y\,:\, y\in S(x)\}$, and by ${\rm  D} (S):= \{x\in X\,:\, S(x)\neq \emptyset\}$  the domain of $S$. When $S$ has a closed graph, we say that $S$  is a closed multifunction. 

Since various types of 
multifunctions arise in a considerable number of models ranging from
mathematical programs, through game theory and to control and design problems, they represent probably  the most developed class of objects in variational analysis. A number of useful regularity properties have been introduced and investigated
(see \cite{Dont-Roc, Roc-Wet}
 and the references therein).  Among them,  the most popular is that of metric regularity  (\cite{AF90, BZ05,    Dont-Roc,  Borisbook1,  penot2012, Roc-Wet,  Schiro, 1RefBorZ, [71], JT6, JT5, JT1,  Jourani-Thibault1,  1RefJT,   Jourani-Thibault4,     ioffe01, RefIo,   Ng,      [36] ,[18],  SIOPT-NT}), 
 the root of which can be traced back to the classical Banach open mapping theorem and the  subsequent fundamental
results of Lyusternik and Graves (\cite{Graves, Lusternik}). 

A multifunction  $F$ is said to be \emph{{ metrically regular}} around $(\bar
{x},\bar{y})\in\gph F$ with modulus $\tau>0$,  whenever  there exist
neighborhoods $ \mathcal{U},  \mathcal{V}$ of
$\bar{x},\bar{y},$ respectively,  such that, for every $(x,y)\in  \mathcal{U}\times \mathcal{V},$%
\begin{equation}\label{MR}
d(x, F^{-1}(y))\leq \tau d(y, F(x)).
\end{equation}
A classical illustration of this concept concerns  the case when $F$ is a bounded linear continuous operator. Then, metric regularity of $F$ amounts to saying that $F$ is surjective.
In terms of the inverse mapping $S := F^{-1}$, property  $(\ref{MR}) $ can be rewritten equivalently as follows:
\begin{equation} \label{aubin}
d(x, S(y))\leq \kappa d(y,y^\prime) \quad \forall y,y^\prime\in\mathcal{V}, \forall x \in S(y^\prime)\cap \mathcal{U}.
\end{equation}
This gives rise to another well known concept called 
 \emph{pseudo-Lipschitz  property},  also called Lipschitz-like property (see\cite{Borisbook1}),  or  \emph{ Aubin property }  (see\cite{aubin84}) at $(\bar y, \bar x)\in \gph S$ .
The concept of \emph{{openness}}  or   \emph{{covering (at a linear rate)}} is also widely used: one says that $S:X\rightrightarrows
Y$ is {\it open at linear rate} $\tau>0$  around $(\bar{x},\bar{y}%
)\in\gph S$ iff  there exist  neighborhoods $ \mathcal{U},  \mathcal{V}$ of
$\bar{x},\bar{y},$ respectively and,  a positive number
$\varepsilon>0$ such that, for every $(x,y)\in\gph S\cap( \mathcal{U}\times
 \mathcal{V})$ and every $\rho\in]0,\varepsilon[,$%
\[
B(y,\rho \tau )\subset S(B(x,\rho)).
\]
We refer to    \cite{aubin84, [20], ioffe01, kruger88, Borisbook1, [71], Roc-Wet, Schiro} and the references therein for  different developments of these notions.
The following relation is well established:
\begin{equation} \text{Metric regularity} \iff \text{Covering} \iff  \text{Aubin property of the inverse}. \end{equation}
Let us also add that in  Banach spaces, similarly  to the classical calculus, one can formulate sufficient (sub)differential characterizations of properties $(\ref{MR})$  and $(\ref{aubin})$ (see, e.g., \cite{ioffe01, kruger88, Borisbook1}). In Asplund spaces (see\cite{Borisbook1, phelps}  for definitions
and characterizations of Asplund spaces), the corresponding characterizations in terms of Fr\'echet  subdifferentials
(\cite{kruger96, kruger03}) or their limiting counterparts ( \cite{kruger85, KrMo, Boris88, Borisbook1}) and the corresponding coderivatives become
necessary and sufficient.

 From the point of view of applications to optimization (sensitivity analysis, convergence analysis
of algorithms, and penalty functions methods), one of the most important regularity properties seems
to be that of error bounds, providing an estimate for the distance of a point from the solution set. This theory 
 was initiated by the pioneering work by Hoffman \cite{[30]}\footnote{It  has been pointed out recently  to the authors by  Hiriart-Urruty that traces of the error bound property were already in  \cite{PCRo}, published in 1951.}. 
 A general classification scheme
of necessary and sufficient criteria for the error bound property is presented in \cite{FHKO,FHKO1}.
 Applications of the theory  of  error bounds  to the investigation of   metric regularity of multifunctions have been recently studied and developed by many authors,  including for instance  \cite{NTT,  RefAz, AzB, SIOPT-NT,   RefNT3, NT09, DuTronStru}.

%The paper is structured  as follows. In  Section 2, we recall  some recent results on error bounds of parametrized systems  and give, sometimes with  some   modifications,   characterizations of metric regularity of multifunctions  given in   \cite{SIOPT-NT, NTT}. In  Section 3, in the context of Asplund spaces, we estimate the strong slope of the function $\varphi_{\mathcal{E}}((x,k),y),$ and give   sufficient conditions as well as  a  point-based condition for metric regularity of the epigraphical multifunction under a coderivative condition. In the last section,  we study   Robinson metric regularity  and Aubin property of a generalized variational system.

 \section{ Metric Regularity of Epigraphical Multifunctions via Error Bounds}
 Let us remind some basic notions  used in  the paper.  Let  $f : X \rightarrow \mathbb{R} \cup
\{+\infty\}$  be a given extended-real-valued function.   As usual, $\dom f := \{x
\in X : f(x) < +\infty\}$ denotes the domain of $f$.
 We recall the concept of error bounds that is one of the most important regularity properties.
 We set
\begin{equation}\label{error}
S:=\{x\in X:\quad f(x)\leq 0\}, \end{equation}
and we  use the symbol $[f(x)]_+$ to denote $\max\{f(x),0\}.$  We say
that  
$f$ satisfies the  {\it an error bound property}  
 iff  there exists  a real $c > 0$  such that \begin{equation}\label{2}
d(x, S) \leq c\big[f(x)]_+ \quad\mbox{ for all}\quad  x\in X.
\end{equation}
For $x_0 \in S$, 
$f$ has a local  
error bound at $x_0$, when there exist reals  $c>0$ and $
\varepsilon > 0$  such that  (\ref{2}) is satisfied for all
$x$ around  $x_0$, i.e.,  in an open  ball  $B(x_0,\varepsilon)$.

Given   a multifunction $F: X\rightrightarrows Y$, we make use  of the lower semicontinuous
envelope $(x,y)\mapsto\varphi_{F}(x,y)$ of the function $(x,y)\mapsto
d(y,F(x)),$ i.e., for $(x,y)\in X\times Y,$
\begin{equation}\label{elie}
\varphi_{F}(x,y):=\liminf_{(u,v)\rightarrow(x,y)}d(v,F(u))=\liminf
_{u\rightarrow x}d(y,F(u)).
\end{equation}
 Recall from  De Giorgi, Marino \& Tosques \cite{RefDMT},
that the
strong slope $\vert\nabla f\vert (x)$ of a lower semicontinuous
function $f$ at $x\in\dom f$ is the quantity defined by
$\vert\nabla f\vert (x)=0$ if $x$ is a local minimum of $f,$ and f
$$\vert\nabla f\vert (x)=\limsup_{y\to x, y\ne x}\frac{f(x)-f(y)}{d(x,y)},$$
otherwise.
For $x\notin \dom f,$  we set  $\vert\nabla f\vert
(x)=+\infty.$  

We now  consider a parametrized inequality system, that
is, the problem of finding $x\in X$ such that
\begin{equation}\label{Para Sys}
f(x,p)\le 0,
\end{equation}
where $f:\; X\times P\to \R\cup\{+\infty\}$ is an
extended-real-valued function, $X$ is a complete  metric space and
$P$ is a topological  space. We denote  by $S(p)$ the set of solutions of  system (\ref{Para Sys}):
$$S(p):=\{x\in X:\quad f(x,p)\le 0\}.$$
The following theorem (\cite[Theorem 2]{NTT}) gives necessary and sufficient conditions for the existence of a local uniform error bound for the parametric system (\ref{Para Sys}).
 %THEOREM 3.1
\begin{theorem}\label{Err Para sys} Let $X$ be a complete metric
space and  $P$ be a topological space. Suppose that the mapping  $f:X\times P\to
\R\cup\{+\infty\}$ satisfies the following conditions for some
$(\bar{x},\bar{p})\in X\times P:$

\item{(a)} $\bar {x}\in S(\bar{p});$

\item{(b)} the mapping $p \mapsto f(\bar {x},p)$ is upper
semicontinuous at $\bar{p};$

\item{(c)} for any $p$ near $\bar{p},$ the mapping $x \mapsto
f(x,p)$ is lower semicontinuous near $\bar {x}.$ \vskip 0.2cm Let
 $\tau>0$ be given. The the following two statements are equivalent:

\item{(i)} There exists a neighborhood $ \mathcal{V} \times  \mathcal{W}\subseteq X\times P$
of $(\bar {x},\bar{p})$ such that for any $p\in \mathcal{W},$ we have $\mathcal{V}\cap
S(p)\not=\emptyset$ and
\begin{equation}\label{ine err}
d(x,S(p))\le \tau[f(x,p)]_+\quad\mbox{for all}\quad (x,p)\in V\times
\mathcal{W}.
\end{equation}
\item{(ii)} There exist a neighborhood $\mathcal{V}\times \mathcal{W}\subseteq X\times P$
of $(\bar {x},\bar{p})$  and a real $\gamma>0$  such that for each
$(x,p)\in \mathcal{V}\times \mathcal{W}$ with $f(x,p)\in (0,\gamma)$ and for any
$\varepsilon>0,$ we can find $z\in X$ such that
\begin{equation}\label{Ine charac}
0<d(x,z)<(\tau+\varepsilon)(f(x,p)-[f(z,p)]_+).
\end{equation}
\end{theorem}
Given metric spaces $X, Y$ and  a topological space $P$, we  next consider the implicit   multifunction 
$:X\times P\rightrightarrows Y$  defined by 
\begin{equation}\label{Imp Multi}
S(y,p):=\{x\in X:\; \; y\in F(x,p)\}.
\end{equation}
Similarly to (\ref{elie}), we use   the lower semicontinuous envelope
$(x,y)\mapsto \varphi_p(x,y)$  of the function $(x,y)\mapsto d(y,
F(x,p))$ for each $p\in P,$ i.e., for $(x,y)\in X\times Y,$
\begin{equation}\label{elie@}
\varphi_p(x,y):=\liminf_{(u,v)\to (x,y)}d(v,F(u,p))=\liminf_{u\to x}d(y,F(u,p)).\end{equation}
From now on, we  will also  use the notation $F_p$ for $F(\cdot,p)$ and $\varphi_p$ for  $\varphi_{F_p}$ and the metric   defined on the cartesian product $X\times Y$  is given by:
$$d((x,y),(u,v))=\max\{d(x,u),d(y,v)\}, \quad (x,y),(u,v)\in X\times Y.$$
%Let us now recall  a result given in
%\cite{NTT}. This result,  which is valid for all $y$ in a neighborhood of $\bar y$, instead of $\bar y$ only,   gives a  necessary  and sufficient  condition for metric regularity of implicit multifunctions. 
%%

The next lemma is useful.
%    LEMMA  %%%%%%%%%%%%%%%%%
\begin{lemma}  We suppose that the set-valued mapping $x \rightrightarrows
F(x,p)$ is  a closed multifunction (i.e., its graph is closed) for any $p$ near $\bar{p}$. Then,
for  each $y\in Y,$
and each $p$ near $\bar{p},$
$$S(y,p)=\{x\in X:\; \varphi_p(x,y)=0\}.$$
\end{lemma}
\begin{theorem}\label{Char2}
Let $X$ be a complete metric space and $Y$ be a metric space. Let
 $P$ be a topological space and suppose that the set-valued mapping
$F:X\times P\rightrightarrows Y$ satisfies the following conditions
for some $(\bar{x},\bar{y},\bar{p})\in X\times Y\times P:$

\item{(a)} $\bar {x}\in S(\bar{y},\bar{p});$

\item{(b)} the multifunction $p \rightrightarrows F(\bar {x},p)$ is  lower
semicontinuous at $\bar{p};$

\item{(c)} for any $p$ near $\bar{p},$ the set-valued mapping $x \rightrightarrows
F(x,p)$ is a closed multifunction (i.e., its graph is closed).
\vskip 0.2cm  Let  $\tau\in (0,+\infty),$ be fixed. Then one
has the following implications: $(i)\Leftrightarrow(ii)\Leftrightarrow
(iii)\Leftarrow(iv).$ Moreover, all the assertions are
equivalent provided that $Y$ is a normed space.

\item{(i)}  There exists a neighborhhood $\mathcal{U}\times \mathcal{V}\times \mathcal{W}\subseteq X\times  P\times Y$ of  $(\bar{x},\bar{y},\bar p)$ such
that $\mathcal{V}\cap S(y,p)\not=\emptyset$ for any $(y,p)\in \mathcal{V}\times \mathcal{W}$ and
$$d(x,S(y,p))\leq\tau d(y,F(x,p))\quad\mbox{for all}\; (x,y,p)\in\mathcal{U}\times \mathcal{V}\times \mathcal{W};$$

\item{(ii)}  There exists a neighborhhood $\mathcal{U}\times \mathcal{V}\times \mathcal{W}\subseteq X\times  P\times Y$ of  $(\bar{x},\bar{y},\bar p)$ such
that $ \mathcal{V}\cap S(y,p)\not=\emptyset$ for any $(y,p)\in  \mathcal{V}\times \mathcal{W}$ and
$$d(x,S(y,p))\leq\tau \varphi_p(x,y)\quad\mbox{for all}\; (x,y,p)\in \mathcal{U}\times \mathcal{V}\times \mathcal{W};$$
\item{(iii)} There exist  a neighborhood $\mathcal{U}\times \mathcal{V}\times \mathcal{W}\subseteq X\times Y\times P$ of  $(\bar{x},\bar y,\bar{p})$ and
a real $\gamma\in(0,+\infty)$ such that for any  $(x,y,p)\in \mathcal{U}\times \mathcal{V}\times \mathcal{W}$ with $y\notin F(x,p) $ and any $\varepsilon>0$, and any sequence $\{x_n\}_{n\in\N}\subseteq X$ converging to
$x$ with
$$\lim_{n\to\infty} d(y,F(x_n,p))=\liminf_{u\to x}d(y,F(u,p))=\varphi_p(x,p),$$ there exists  a sequence $\{u_n\}_{n\in \N}\subseteq X$ with
$\liminf_{n\to\infty} d(u_n,x)>0$ such that \begin{equation}\label{Inchar2}
\limsup_{n\to\infty}\frac{d(y,F(x_n,p))-d(y,F(u_n,p))}{d(x_n,u_n)}>\frac{1}{\tau+\varepsilon};
\end{equation}

\item{(iv)} There exist  a neighborhood $\mathcal{U}\times \mathcal{V}\times \mathcal{W}\subseteq X\times P\times Y$ of  $(\bar{x},\bar{p},\bar{y})$ and a real $\gamma>0$ such that 
\begin{equation}\label{Cond Slope}
|\nabla\varphi_{p}(\cdot,y)|(x)\geq \frac{1}{\tau}\quad\mbox{for all}\; (x,y,p)\in
\mathcal{U}\times \mathcal{V}\times \mathcal{W}\quad\mbox{with}\; \varphi_p(x,y)\in (0,\gamma).
\end{equation}
\end{theorem}
%%%%%%%%%%%%%%%
 \textit{Proof.} The implications $(ii)\Rightarrow(i)$ and $(iv)\Rightarrow(iii)$ are obvious.
 For $(i)\Rightarrow (iii),$ let $\mathcal{U}\times\mathcal{V}\times \mathcal{W}$ be an  open neighborhood
 of $(\bar{x},\bar{y},\bar p)$ such that $\gph F(\cdot,p)$ is closed
 for $p\in \mathcal{W}$ and  
  $$d(x,S(y,p))\leq \tau d(y, F(x,p))\;\;\forall (x,y,p)\in \mathcal{U}\times\mathcal{V}\times \mathcal{W}.$$
Let $(x,y,p)\in \mathcal{U}\times\mathcal{V}\times \mathcal{W}, \  y\notin F(x,p)$  and $\varepsilon>0.$ Let $\{x_n\}_{n\in\N}$
be a sequence converging to $x.$ When $n$ is sufficiently large, say
$n\geq n_0,$ then $x_n\in \mathcal{U}$ as well as $ y\notin F(x_n,p).$
Hence $d(x_n,S(y,p))\leq \tau d( y, F(x_n,p)).$ For each
$n\geq n_0,$ pick $u_n\in S(y,p)$ such that
$d(x_n,u_n)<(1+\varepsilon/2\tau)d(x_n,S(y,p)).$ 
We claim that $\liminf_{n\to\infty} d(u_n,x)>0$. 
Otherwise, there would exist some subsequence, $\{u_{n_k}\}_{k\in \N}$ converging to $x$ such that 
$u_{n_k}\in S(y, p), i.e.,  y \in F(u_{n_k},p)$.
Then, since $ F(\cdot,p)$  is graph-closed this would imply 
 $y \in F(x,p)$, a contradiction.
 Moreover for all $n\geq n_0,$
$$d(x_n,u_n)<(1+\varepsilon/2\tau)d(x_n,S( y,p))\leq (\tau+\varepsilon/2)[d(y,F(x_n,p))-d(y, F(u_n,p))].$$
This shows that (\ref{Inchar2}) holds. 
\vskip 0.5mm

 For $(iii)\Rightarrow (ii).$ Since the multifunction $p\rightrightarrows
F(\bar x,p)$ is assumed to be lower semicontinuous at $\bar p,$ then
the function $(p,y)\mapsto d(y,F(\bar x,p))$ is upper
semicontinuous at $(\bar p,\bar y)$ (see, e.g.,  in \cite[Cor. 20 ]{AuE}).  Therefore,
$$ \limsup_{(p,y)\to(\bar{p},\bar y)}\varphi_p(\bar x, y)\le\limsup_{(p,y)\to(\bar{p},\bar y)}d( y, F(\bar x,p)) \le d(\bar y, F(\bar x,\bar p))=\varphi_{\bar p}(\bar x,\bar y).$$
That is,  the function
$(p,y)\mapsto\varphi_p(\bar x, y)$ is upper semicontinuous at $(\bar
p,\bar y),$  and therefore, by virtue of Theorem \ref{Err Para sys}, it suffices
to observe  that statement  (ii)  of Theorem \ref{Err Para sys} is
verified. Indeed, let $(x,y,p)\in \mathcal{U}\times \mathcal{V}\times\mathcal{W}$ with $y\notin F(x,p) $ and
$\varphi_p(x,y)<\gamma$ and let $\varepsilon\in(0,1)$ be given. Let
$\{x_n\}_{n\in\N}$ be a sequence converging to $x$ with
$$\lim_{n\to\infty}d(y, F(x_n,p))=\varphi_p(x,y)=\liminf_{u\to x}d(y,F(u,p)).$$
Then, $x_n\notin F^{-1}_p(y)$, i.e., $y\notin F(x_n,p)$ when $n$ is sufficiently large, say $n\geq n_0$.
By (iii), we consider  a sequence $\{u_n\}_{n\in\N}$ with
$\liminf_{n\to\infty}d(u_n,x)>0$ such that
$$\limsup_{n\to\infty}\frac{d(y,F(x_n,p))-d(y,F(u_n,p))}{d(x_n,u_n)}>\frac{1}{\tau+\varepsilon}.$$
Pick $\delta\in(0,\liminf_{n\to\infty}d(u_n,x))$. Then, take  an index $n_1\geq n_0$ such that for all $n\geq n_1$, we have
$$d(x_n,u_n)\geq \delta; d(x_n,x)<\varepsilon\delta; d( y,F(x_n,p)<\varphi_p(x,y)+\frac{\varepsilon}{\tau+\varepsilon}d(x_n,u_n)$$
and
$$d(x_n,u_n)<(\tau+\varepsilon)(d(y,F(x_n,p))-d(y,F(u_n,p))).$$
Hence, 
$$d(x_n,u_n)<(1-\varepsilon)^{-1}(\tau+\varepsilon)(\varphi_p(x,y)-\varphi_p(u_n,y)).$$
It follows that for all $n\geq n_1$,
\begin{align*}
d(x,u_n)\le (1+\varepsilon)d(x_n,u_n)\\
&<(1-\varepsilon)^{-1}(\tau+\varepsilon)(1+\varepsilon)(\varphi_p(x,y)-\varphi_p(u_n,y))\\
&<(\tau+\varepsilon)(\varphi_p(x,y)-\varphi_p(u_n,y))
\end{align*}
and statement (ii) of Theorem \ref{Err Para sys} follows directly. So, the implication  $(iii)\Rightarrow (ii)$ is now proved.
\vskip 0.5cm
When $Y$ is normed space, $(i) \Rightarrow (iv)$ follows from the converse part of  \cite[Theorem 5]{NTT}) by noting that $S(y,p)=F_p^{-1}(y)$. So, we have that all assertions are equivalent when $Y$ to be normed space.\\
The proof is complete. \qed
%\hfill{$\Box$}

%\bibitem{AuE}  Aubin, J.-P., Ekeland I., {\em Applied Nonlinear Analysis}, John Wiley \& Sons 1984.
%\bibitem {NTT}Ngai, H.V., Th\'{e}ra, M., Nguyen, H.T., Implicit
%multifunction theorems in complete metric spaces, {\em Math. Program.}, to appear.
%%%%%%%%%%%%%%%%%%%%%%%%%%%%%%%%%%%%%
%%%%%%%%%%%%%%%%%%%%%%%%%%%%%%%%%%%%%
%%%%%%%%%%%%%%%%%%%%%%%%%%%%%%%%%%%%%
Given two  multifunctions  
$F, G:X\rightrightarrows Y$,  ($Y$ is a normed linear space) we define a new multifunction  $\mathcal{E}_{(F,G)}: X\times Y\rightrightarrows Y $   %associated  with  $F$ and $ G$
 by setting
\[
\mathcal{E}_{(F,G)}(x,k)=\left\{
\begin{array}
[c]{ll}%
F(x)+k, & \text{if $k\in G(x),$}\\
\emptyset, & \text{otherwise.}%
\end{array}
\right.
\]
When one of the multifunctions is a cone,  $\mathcal{E}_{(F,G)}$was called \textit{epigraphical} by Durea and Strugariu \cite{Dustru}.

For given
$y\in Y,$ we set
\begin{equation}\label{Imp Multi@}
\S_{\mathcal{E}_{(F,G)}}( y):=\{(x,k)\in X\times Y:\; \;  y\in \mathcal{E}_{(F,G)}(x,k)\}.
\end{equation}
The lower semicontinuous envelope
$((x,k),y)\mapsto \varphi_{\mathcal{E}}((x,k),y)$  of the distance function $d(y, \mathcal{E}_{(F,G)}(x,k))$ is defined for $(x,k,y)\in X\times Y\times Y$ by
\begin{align*}
&\varphi_{\mathcal{E}}((x,k),y):=\liminf_{(u,v,w)\to (x,k,y)}d(w,\mathcal{E}_{(F,G)}(u,v)).
 \end{align*}
Let us recall that a multifunction $G:X\rightrightarrows Y$ is lower semicontinuous at $(x,y)\in\gph G$,  if for any sequence $\{x_n\}_{n\in N}$ converging to $x$,  we can provide  a sequence $\{y_n\}_{n\in\N}$ converging to $y,$ with $y_n\in G(x_n).$ 
\begin{lemma}\label{tron}

If   $G$ has closed graph then 
\begin{align*}
\varphi_{\mathcal{E}}((x,k),y)
=\left\{
\begin{array}
[c]{ll}%
\liminf\limits_{\gph G\owns(u,v)\to (x,k)}d(y,F(u)+v),  & \text{if $k\in G(x)$}\\
+\infty,  & \text{otherwise.}%
\end{array}
\right.\\
 \end{align*}
 Moreover,  if in addition,   $G$  is lower semicontinuous at $(x,k)\in \gph G$,  then the following representation holds:
\begin{align*}
\varphi_{\mathcal{E}}((x,k),y)
=\left\{
\begin{array}
[c]{ll}%
\liminf\limits_{u\rightarrow x}d(y,F(u)+k),  & \text{if $k\in G(x)$}\\
+\infty,  & \text{otherwise.}%
\end{array}
\right.\\
 \end{align*}
\end{lemma}
\begin{proof}
For the first equality, if $k\notin G(x)$,  since $G$ has closed graph one has $\varphi_{\mathcal{E}}((x,k),y)=\infty$. Otherwise, we have 
\begin{align*}
\varphi_{\mathcal{E}}((x,k),y 
&=\liminf_{(u,v,w)\to (x,k,y)}d(w,\mathcal{E}_{(F,G)}(u,v))\\
%& =\liminf_{(u,v)\to (x,k),\,v\in G(u)}d(y,\mathcal{E}_{(F,G)}(u,v))\\
&=\liminf_{\gph G\owns(u,v)\to (x,k)}d(y,F(u)+v).
 \end{align*}
\textbf{Claim} Let $G:X\rightrightarrows Y$ be lower semicontinuoous at $(x,k)\in \gph G$. Then for each $y\in Y $we have 
$$\liminf_{\gph G\owns(u,v)\to (x,k)}d(y,F(u)+v)=\liminf_{u\to x}d(y,F(u)+k).$$
For simplicity set $A:=\;\liminf_{\gph G\owns(u,v)\to (x,k)}d(y,F(u)+v)$ and $B:= \;\liminf_{u\to x}d(y,F(u)+k).$
First let us prove that $A\geq B.$ Indeed, let $\{(u_n, v_n)\}_{n\in \N}$ be a sequence in $\gph G$ such that $(u_n,v_n) \to (x,k)$ as $n\to +\infty$ and $\lim_{n\to +\infty} d(y, F(u_n)+ v_n) = A.$ Then, 
\begin{align*}
B\leq\liminf_{n\to +\infty}d(y, F(u_n) +k)&\leq \liminf_{n\to +\infty}[d(y, F(u_n) + v_n)+\Vert v_n-k\Vert]\\
&=\lim_{n\to +\infty} d(y,F(u_n)+v_n)= A.
\end{align*}
On the other hand, to prove that $A\leq B$, pick any sequence $\{u_n\}_{n\in\N}$ converging to $x$ such that $\lim_{n\to +\infty} d(y, F(u_n) + k) = B.$  As $G$ is lower semicontinuous at $(x,k)$, we find a sequence $\{v_n\}_{n\in\N}$ converging to $k$ such that $(u_n,v_n) \in \gph G$ for each $n\in\N.$
Hence, 
\begin{align*}
A&\leq \liminf_{n\to +\infty} d(y, F(u_n) +v_n)\\
&\leq \liminf_{n\to +\infty}[d(y, F(u_n)+k +\Vert k-v_n\Vert ]\\
&\leq \lim_{n\to +\infty} d(y, F(u_n) + k) = B.
\end{align*}
The claim is proved. From the claim, the fact that $ \varphi_{\mathcal{E}}((x,k),y)=\liminf_{u\to x}d(y,F(u)+k)$ follows immediately. \qed
 \end{proof}
  %%%%%%%% REMARK %%%%%%%%%%%%%%%%%%%%%%%%
\begin{remark}
\begin{enumerate}
\item Since we suppose that $G$ is both graph-closed and lower semicontinuous, it is continuous in finite dimension (see, \cite[Theorem 5.7 page 158]{Roc-Wet}).
\item
The lower semicontinuity of $G$ is necessary to obtain the last formula in Lemma \ref{tron} as shows the next example\footnote{We  would like to thank one of the referees for pointing us  this example.}: take  $F, G : [0,1 ]  \rightrightarrows  \R$ be defined by  $F(0)= \{0\}, F(x)=1 \; \text{if}\; x\in (0,1]$ and $G(0) =\{0,1\}, G(x) = \{0\},  \; \text{if}\; x\in(0,1].$ Note that $G$ has a closed graph but is not lower semicontinuous at $(0, 1)\in \gph G$ and remark that $$\liminf_{u\to 0} d(3, F(u)+1)= 1 \; \text{while} \;  \varphi_{\mathcal{E}}((0,3),1)=  \liminf_{(u,v,w)\to (0,1,3)}d(w,\mathcal{E}_{(F,G)}(u,v))=2.$$
\end{enumerate}
\end{remark}
%%%%%%%%%%%%%%%%% LEMMA %%%%%%%%%%
The next lemma is useful.
\begin{lemma}\label{dame}  Assume that $F: X\rightrightarrows Y$ and $G: X\rightrightarrows Y$ be closed multifunctions.  Then, the epigraphical multifunction $\mathcal{E}_{(F,G)}$ has a closed graph, and 
for each $y\in Y,$
%%%%%%%%%%%%%%%%% LEMMA %%%%%%%%%%
\begin{equation}{\label{lemma}}
\S_{\mathcal{E}_{(F,G)}}( y)=\{(x,k)\in X\times Y:\quad \varphi_{\mathcal{E}}((x,k),y)=0\}=\{(x,k)\in X\times Y:\quad k\in G(x), y\in F(x)+k\}.
\end{equation}
\end{lemma}
\begin{proof}
 Observe that,   if $F: X\rightrightarrows Y$ and $G: X\rightrightarrows Y$ are closed multifunctions, then so is the epigraphical multifunction $\mathcal{E}_{(F,G)}.$ 

Let us prove (\ref{lemma}).
Obviously, for each $y\in Y, $ if $(x,k)\in \S_{\mathcal{E}_{(F,G)}}(y)$,  then $\varphi_{\mathcal{E}}((x,k),y)=0.$ 
Conversely,  suppose that $\varphi_{\mathcal{E}}((x,k),y)=0.$  Then, $k\in G(x)$ and there exists a sequence $\{(x_n,k_n)\}\to (x,k), k_n\in G(x_n)$ such that $d(y,F(x_n)+k_n)\to 0.$ Then, there exists $z_n\in F(x_n)$ such that $z_n+k_n\rightarrow y$. It follows, $z_n\rightarrow y-k$.  Since $F$ is graph-closed, one has that $y-k\in F(x),$ i.e., $y\in F(x)+k.$  Hence,  $(x,k)\in \S_{\mathcal{E}_{(F,G)}}(y)$  establishing the proof.\qed 
\end{proof}
By virtue of  Lemma \ref{dame},  we adapt Theorem \ref{Char2} to the multifunction  ${\mathcal{E}_{(F,G)}}$.
\begin{lemma}\label{strophi}
Let $X$ be a complete metric
space, let $Y$ be a Banach space and let $F, G:X\rightrightarrows Y$ be  closed multifunctions. Suppose that  $(\bar{x},\bar{k},\bar{y})\in X\times Y\times Y$ such that $\bar{y}\in F(\bar{x})+\bar{k}, \bar{k}\in G(\bar{x})$.

 Let  $\tau\in ]0,+\infty[,$ be fixed. Then,  the following
statements are equivalent:
\item{(i)}  There exists a neighborhood \;$\mathcal{U}\times \mathcal{V}\times \mathcal{W}\subseteq X\times  Y\times Y$ of  $(\bar{x},\bar{k},\bar{y})$ such
that  \;$(\mathcal{U}\times\mathcal{V}) \cap \S_{\mathcal{E}_{(F,G)}}(y)\not=\emptyset$ for any $y\in \mathcal{W}$  and
$$d((x,k),\S_{\mathcal{E}_{(F,G)}}(y))\leq\tau \varphi_{\mathcal{E}}((x,k),y)\quad\mbox{for all}\quad (x,k,y)\in \mathcal{U}\times \mathcal{V}\times \mathcal{W};$$
\item{(ii)} There exist  a neighborhood \;$\mathcal{U}\times \mathcal{V}\times \mathcal{W}\subseteq X\times Y\times Y$ of  $(\bar{x},\bar{k},\bar{y})$ and
a real $\gamma\in]0,+\infty[$ such that,  for any  $(x,k,y)\in \mathcal{U}\times \mathcal{V}\times \mathcal{W}$ with $y\notin F(x)+k, k\in G(x) $ and
$\varphi_{\mathcal{E}}((x,k),y)<\gamma$,   any $\varepsilon>0$, and any sequences $\{x_n\}_{n\in\N}\subseteq X$ converging to
$x$, $\{k_n\}_{n\in\N}\subseteq Y$ converging to
$k$, $k_n\in G(x_n)$  with
$$ \lim_{n\to\infty} d(y-k_n,F(x_n))=
\liminf_{\gph G \owns (u,v)\to (x,k)}d(y-v,F(u)),$$ there exist  sequences $\{u_n\}_{n\in \N}\subseteq X, \  \{z_n\}_{n\in \N}\subseteq Y$  with $(u_n,z_n)\in\gph G$ and 
$ \liminf_{n\to\infty} d((u_n,z_n),(x,k))>0$ such that 
\begin{equation}\label{Inchar2}
\limsup_{n\to\infty}\frac{d(y-k_n,F(x_n))-d(y-z_n,F(u_n))}{d((x_n,u_n),(k_n,z_n))}>\frac{1}{\tau+\varepsilon};
\end{equation}
\item{(iii)} there exist a neighborhood \;$\mathcal{U}\times \mathcal{V}\times \mathcal{W}$  of
$(\bar{x},\bar{k},\bar{y})$ and    a real  $\gamma>0$  such that 
$$\vert\nabla\varphi_{\mathcal{E}}((\cdot,\cdot),y)\vert (x,k)\geq \frac{1}{\tau}\; {\rm  for \;all} \; (x,k,y)\in \mathcal{U}\times \mathcal{V}\times \mathcal{W} \;{\rm with}\; 
\varphi_{\mathcal{E}}((x,k),y)\in ]0,\gamma[.$$
\end{lemma}
%PROPOSITION
\begin{proposition} \label{Slope E} Let $X$ be a complete metric
space, $Y$ be a Banach space and let $F, G:X\rightrightarrows Y$ be  closed multifunctions. Suppose that  $(\bar{x},\bar{k},\bar{y})\in X\times Y\times Y$ be such that $\bar{y}\in F(\bar{x})+\bar{k}, \bar{k}\in G(\bar{x})$. Consider the following statements:
\item{(i)} there exist a neighborhood \;$\mathcal{U}\times \mathcal{V}\times \mathcal{W}$ of
$(\bar{x},\bar{k},\bar{y})$  and $\tau>0$ such that $$ d((x,k),\S_{\mathcal{E}_{(F,G)}}(y))\leq \tau\varphi_{\mathcal{E}}((x,k),y)\quad \mbox{for all}\quad (x,k,y)\in \mathcal{U}\times \mathcal{V}\times \mathcal{W}; $$
 \item{(ii)} there exist a neighborhood \;$\mathcal{U}\times \mathcal{V}\times \mathcal{W}$ of $(\bar{x},\bar{k},\bar{y})$  and $\tau>0$  such that
\begin{equation}\label{star@}
 d(x,(F+G)^{-1}(y))\leq \tau d(y,F(x)+G(x)\cap \mathcal{V})\quad \mbox{for all}\quad (x,y)\in\mathcal{U}\times \mathcal{W};
\end{equation}
\item{(iii)} there exist a neighborhood \;$\mathcal{U}\times \mathcal{V}\times \mathcal{W}$ of $(\bar{x},\bar{k},\bar{y}-\bar{k})$  and $\varepsilon,\tau>0$ such that, for every $(x,k,z)\in \mathcal{U}\times \mathcal{V}\times \mathcal{W},   k\in G(x), z\in F(x),$ and 
 $\rho\in ]0,\varepsilon[,$
$$B(k+z,\rho\tau^{-1})\subset  (F+G)(B(x,\rho)).$$
Then one has the following implications: $(i)\Rightarrow (ii)\Leftrightarrow (iii).$
\end{proposition}
\begin{proof}
 For $(i)\Rightarrow (ii).$ By (i), there exist $\delta_1,\delta_2,\delta_3>0$ such that,  for every $\varepsilon>0$ and for every  $(x,k,y)\in B(\bar{x},\delta_1)\times [B(\bar{k},\delta_2)\cap G(x)]\times B(\bar{y},\delta_3),$ there is $(u,z)\in X\times Y$ with $y\in F(u)+z, z\in G(u)$ such that  
$$ d((x,k),(u,z))< (1+\varepsilon)\tau\varphi_{\mathcal{E}}((x,k),y).$$ 
Consequently,
$$ d(x,u)\leq \max\{d(x,u), \|k-z\|\}< (1+\varepsilon)\tau d(y,F(x)+k).$$
Noting  that $y\in F(u)+ G(u),$ i.e., $u\in (F+G)^{-1}(y)$,  it follows that 
$$d(x,(F+G)^{-1}(y))<(1+\varepsilon)\tau d(y,F(x)+k).$$
In conclusion, we have that 
$$d(x,(F+G)^{-1}(y))< (1+\varepsilon)\tau d(y,F(x)+G(x)\cap B(\bar{k},\delta_2))\quad \mbox{for all}\quad (x,y)\in B(\bar{x},\delta_1)\times B(\bar{y},\delta_3).$$
Hence, taking the limit as $\varepsilon>0$ goes to $0$ yields the desired  conclusion. 

For $(ii)\Rightarrow (iii).$ 
Suppose that (ii) holds for the neighborhood $B(\bar{x},\delta_1)\times B(\bar{k},\delta_2)\times B(\bar{y},\delta_3)$ with $\delta_1,\delta_2,\delta_3>0$ and $\tau>0.$ Choose $\rho_1=\delta_1,\rho_2=1/4\min\{\delta_2,\delta_3\},\rho_3=1/4\delta_3, \varepsilon<\tau\delta_3/2.$ \\
Then, for  $(x,k,z)\in B(\bar{x},\rho_1)\times B(\bar{k},\rho_2)\times B(\bar{y}-\bar{k},\rho_3), k\in G(x), z\in F(x),$ we take $y\in B(k+z,\rho\tau^{-1}).$\\
Consequently, $$\Vert y-k-z\Vert<\rho\tau^{-1},$$
and 
$$
 \begin{array}{lll}
 \Vert y-\bar{y}\Vert
& \le \Vert y-k-z\Vert+\Vert k-\bar{k}\Vert+\Vert \bar{k}-\bar{y}+z\Vert,\\
&<\rho\tau^{-1}+\rho_2+\rho_3,\\
&<\varepsilon\tau^{-1}+\delta_3/4+\delta_3/4,\\
&<\delta_3/2+\delta_3/2=\delta_3.\\
\end{array}
$$
Therefore, we have that $$d(y,F(x)+G(x)\cap B(\bar{k},\delta_2))\le \Vert y-k-z\Vert<\rho\tau^{-1}.$$
Hence, $$d(x,(F+G)^{-1}(y))<\tau\rho\tau^{-1}=\rho.$$  Let $\gamma>0$ with $d(x,(F+G)^{-1}(y))+\gamma<\rho.$
Find $u\in (F+G)^{-1}(y)$, i.e., $y\in (F+G)(u)$ such that $$d(x,u)<d(x,(F+G)^{-1}(y))+\gamma.$$ Thus, $d(x,u)<\rho.$
It follows that $$y\in (F+G)(B(x,\rho)).$$
 For $(iii)\Rightarrow (ii).$ Suppose that (iii) holds for the neighborhood $B(\bar{x},\rho_1)\times B(\bar{k},\rho_2)\times B(\bar{y},\rho_3)$ with $\rho_1,\rho_2,\rho_3>0$ and $\tau>0,\varepsilon>0.$  \\
Take $\rho_1,\rho_3$ smaller  if neccesary and consider a positive real $\eta$ sufficiently small  so that the quantity $\rho:=\tau d(y,F(x)+G(x)\cap B(\bar{k},\rho_2))+\eta$ satisfies the conclusion of (iii) together with $y\in B(k+z,\rho\tau^{-1}).$ Then, there is a $u\in B(x,\rho)$ such that $y\in (F+G)(u)$, that is, $u\in (F+G)^{-1}(y).$ \\
Thus, 
$$d(x,(F+G)^{-1}(y))\leq d(x,u)<\rho=\tau d(y,F(x)+G(x)\cap B(\bar{k},\rho_2))+\eta.$$
Since $\eta>0$ is arbitrary,  the proof is complete.\qed 
\end{proof}

The next result  gives  conditions for the sum of two  metrically regular mappings $F, G$ to remain  metrically regular.  Before stating this  result, we need to recall the   so-called ``locally sum-stable" property  introduced in \cite{Dustru}. 
%%%%%%%%%%%  DEFINITION   %%%%%%%%%%%%%%
 \begin{definition}
Let $F, G: X\rightrightarrows Y$ be two multifunctions and $(\bar{x},\bar{y}, \bar{z})\in X\times Y\times Y$ such that $\bar{y}\in F(\bar{x}),\bar{z}\in G(\bar{x}).$ We say that the pair  $(F,G)$ is locally sum-stable around $(\bar{x},\bar{y}, \bar{z})$ iff for every $\varepsilon>0,$ there exists $\delta>0$ such that, for every $x\in B(\bar{x},\delta)$ and every $w\in (F+G)(x)\cap  B(\bar{y}+\bar{z},\delta),$ there are $y\in F(x)\cap  B(\bar{y},\varepsilon)$ and $z\in G(x)\cap  B(\bar{z},\varepsilon)$ such that $w=y+z.$
\end{definition}
A simple case which ensures the local sum-stability of  $(F,G)$ is as follows.
%PROPOSITION
\begin{proposition}\label{usc-sum}
Let $F:X\rightrightarrows Y, G:X\rightrightarrows Y$ be two multifunctions and $(\bar{x},\bar{y}, \bar{z})\in X\times Y\times Y$ such that $\bar{y}\in F(\bar{x}),\bar{z}\in G(\bar{x}).$ If $G(\bar{x})=\{\bar{z}\}$ and $G$ is upper semicontinuous at $\bar{x},$ then  the pair  $(F,G)$ is locally sum-stable around $(\bar{x},\bar{y}, \bar{z}).$ 
\end{proposition}
\begin{proof} Since  $G$ is upper semicontinuous at $\bar{x},$ for every $\varepsilon>0$ there exists $\delta>0$ such that $$G(x)\subset G(\bar{x})+B(0,\varepsilon/2) = \bar{z}+B(0,\varepsilon/2)=B(\bar{z},\varepsilon/2), \quad \text{for all} \;x\in B(\bar{x},\delta).$$
Set  $$\eta:=\min\{\delta,\varepsilon/2\}$$ and take $x\in B(\bar{x},\eta)$ and $w\in (F+G)(x)\cap B(\bar{y}+\bar{z},\eta)$. Then, there are $y\in F(x), z\in G(x)$ such that
$$w=y+z \;\text{and}\; w\in B(\bar{y}+\bar{z},\eta).$$ 
Clearly, $z\in B(\bar{z},\varepsilon/2)\subset B(\bar{z},\varepsilon).$\\ Moreover, $$\Vert y-\bar{y}\Vert=\Vert w-z-\bar{y}\Vert\le \Vert w-\bar{y}-\bar{z}\Vert+\Vert z-\bar{z}\Vert<\eta+\varepsilon/2\le \varepsilon/2+\varepsilon/2=\varepsilon.$$
Consequently, $$w=y+z, y\in F(x)\cap B(\bar{y},\varepsilon),z\in G(x)\cap B(\bar{z},\varepsilon).$$ 
Hence we have established that $(F,G)$ is locally sum-stable around $(\bar{x},\bar{y}, \bar{z}).$\qed 
\end{proof}
%PROPOSITION
 \begin{proposition}\label{me-sum} Let $X$ be a complete metric
space, $Y$ be a Banach space and let $F, G:X\rightrightarrows Y$ be  closed multifunctions. Suppose that  $(\bar{x},\bar{k},\bar{y})\in X\times Y\times Y$ is such that $\bar{y}\in F(\bar{x})+\bar{k}, \bar{k}\in G(\bar{x})$.\\
If  the pair  $(F,G)$ is locally sum-stable around $(\bar{x}, \bar{y}-\bar{k}, \bar{k} )$ and there exist a neighborhood $\mathcal{U}\times \mathcal{V}$ of $(\bar{x},\bar{y})$  and $\tau,\theta>0$    such that
\begin{equation}\label{star}
 d(x,(F+G)^{-1}(y))\leq \tau d(y,F(x)+G(x)\cap  B(\bar{k},\theta) )\quad \mbox{for all}\quad (x,y)\in \mathcal{U}\times \mathcal{V},
\end{equation}
 then $F+G$ is metrically regular around $(\bar{x},\bar{y})$ with modulus $\tau.$

 As a result, if $G$ is upper semicontinuous at $\bar{x}$ and $G(\bar{x})=\{\bar{k}\},$ then $F+G$ is metrically regular around $(\bar{x},\bar{y})$ with modulus $\tau.$
 \end{proposition}
 \begin{proof}
  Suppose that (\ref{star}) holds for every $ (x,y)\in B(\bar{x},\delta_1) \times  B(\bar{y},\delta_2)$  for some $ \delta_1,\delta_2>0.$
 Since $(F,G)$ is locally sum-stable around $(\bar{x}, \bar{y}-\bar{k}, \bar{k}),$  there exists $\delta>0$ such that,  for every $x\in B(\bar{x},\delta)$ and every $w\in (F+G)(x)\cap  B(\bar{y},\delta),$ there are $y\in F(x)\cap  B(\bar{y}-\bar{k},\theta)$ and $z\in G(x)\cap  B(\bar{k},\theta)$ such that $w=y+z.$\\
Taking $\delta$ smaller if necessary, we can assume that $\delta<\delta_1.$ Fix $(x,y)\in B(\bar{x},\delta/2) \times  B(\bar{y},\delta/2).$ 
We consider two following cases:

Case 1. $d(y,F(x)+G(x))<\delta/2.$  Fix $\gamma>0$ , small enough in order to have   $$d(y,F(x)+G(x))+\gamma<\delta/2,$$ and take  $t\in F(x)+G(x)$ such that $\Vert y-t\Vert<d(y,F(x)+G(x))+\gamma.$  Hence we have $\Vert y-t\Vert<\delta/2,$
and since  we also have $\Vert y-\bar{y}\Vert<\delta/2, $ this yields  $$\Vert t-\bar{y}\Vert\le \Vert y-t\Vert+\Vert y-\bar{y}\Vert<\delta/2+\delta/2=\delta.$$ 
It follows that $$t\in [F(x)+G(x)]\cap B(\bar{y},\delta).$$ Since $(F,G)$ is locally sum-stable   around $(\bar x, \bar y-\bar k,\bar k)$,     there are $y\in F(x)\cap  B(\bar{y}-\bar{k},\theta)$ and $z\in G(x)\cap  B(\bar{k},\theta)$ such that $$t=y+z.$$
Consequently,  $$t\in F(x)\cap  B(\bar{y}-\bar{k},\theta)+G(x)\cap  B(\bar{k},\theta)\subset F(x)+G(x)\cap  B(\bar{k},\theta).$$
Therefore, $$d(y,F(x)+G(x)\cap  B(\bar{k},\theta) )\le \Vert y-t\Vert,$$
from which we derive
 $$d(y,F(x)+G(x)\cap  B(\bar{k},\theta) )\le d(y,F(x)+G(x)) + \gamma,$$ and therefore, as $\gamma$ is arbitrarily small, we obtain  that
 $$d(y,F(x)+G(x)\cap  B(\bar{k},\theta) )\le d(y,F(x)+G(x)).$$
By (\ref{star}), one gets that $$d(x,(F+G)^{-1}(y))\leq \tau d(y,F(x)+G(x)).$$
Since $(x,y)$ is arbitrary in $B(\bar{x},\delta/2) \times  B(\bar{y},\delta/2)$, this yields  $$d(x,(F+G)^{-1}(y))\leq \tau d(y,F(x)+G(x)),$$
$\text{for\; all}\;(x,y)\in B(\bar{x},\delta/2) \times  B(\bar{y},\delta/2).$\\
Case 2.  If  $d(y,F(x)+G(x))\geq \delta/2.$  Choose  $\delta$   sufficiently small so that $\tau\delta/4<\delta_1.$ For every $(x,y)\in B(\bar{x},\tau\delta/4) \times  B(\bar{y},\delta/4) $ and  any $\varepsilon>0,$ by (\ref{star}), there exists $u\in (F+G)^{-1}(y)$ such that $$d(\bar{x},u)<(1+\varepsilon)\tau d(y,F(\bar{x})+G(\bar{x}))\le (1+\varepsilon)\tau\Vert y-\bar{y}\Vert<(1+\varepsilon)\tau\delta/2\le (1+\varepsilon)\tau/2d(y,F(x)+G(x)).$$ So,
\begin{align*}
 d(x,u)\le d(x,\bar{x})+d(\bar{x},u)\\
 &<\tau\delta/4+(1+\varepsilon)\tau/2d(y,F(x)+G(x))\\
 &<\tau/2d(y,F(x)+G(x))+(1+\varepsilon)\tau/2d(y,F(x)+G(x)).\\
 \end{align*}
Taking the limit as $\varepsilon>0$ goes to $0$, it follows that $$d(x,(F+G)^{-1}(y))\leq \tau d(y,F(x)+G(x)).$$ So, $$d(x,(F+G)^{-1}(y))\leq \tau d(y,F(x)+G(x)),$$ $\text{for all}\;(x,y)\in B(\bar{x},\tau\delta/4) \times  B(\bar{y},\delta/4).$
The proof is complete.\qed 
\end{proof}
 
The following example shows that the sum of a metrically regular set-valued mapping and a pseudo-Lipschitz  one is  not generally metrically  regular without the sum-stability (see \cite{Dustru} for a similar example on the sum of two pseudo-Lipschitz   set-valued mappings).
\begin{example} Let $F, G:\R\rightrightarrows \R$ be given  as
\[
F(x):=
\begin{cases}
[-x, +\infty[, & \text{if $x\in [0,+\infty[$}\\
\{-1\}, & \text{otherwise}%
\end{cases}
\]
and
$$ G(x):=\{0,1\}, \;x\in\R.$$
Then, obviously, $F, G$ are closed multifunctions and, it is easy to see that $F$ is metrically regular around $(0,0)$ and $G$ is pseudo-Lipschitz around $(0,0)$.
However, $(F,G)$ is not sum-stable around $(0,0,0)$ and $F+G$  fails to be  metrically regular around $(0,0)$.
\end{example}
\begin{proof}
 Indeed, we have that
 
 \[
(F+G)(x)=
\begin{cases}
[-x, +\infty[, & \text{if} x\in [0,+\infty[\\
+\infty & \text{otherwise.}
\end{cases}
\]
and
\[
(F+G)^{-1}(x)=
\begin{cases}
]-x, +\infty[, & \text{if $x\in ]-\infty,0[\setminus\{-1\}$}\\
\R, & \text{if $x=0$}\\
]-\infty,0]\cup ]1,+\infty[, & \text{if $x=-1$}\\
]0, +\infty[\cup\{1\} , & \text{if $x\in ]0,+\infty[.$}\\
\end{cases}
\]
Suppose that $F+G$ is metrically regular around $(0,0)$, then there exist $\tau>0$ and $0<\delta<\min\{1, \tau^{-1}\}$ such that, for every $(x,y)\in ]-\delta,\delta[\times ]-\delta,\delta[$, one has 
\begin{equation}\label{HDHM}
d(x,(F+G)^{-1}(y))\le \tau d(y,(F+G)(x)).\end{equation}
Consider $x:=-\delta/2$ and $y:=-\delta^{2}/2$. Then, $x\in ]-\infty,0[$, $y\in ]-\infty,0[$ and $$(F+G)(x)=\{-1,0\}, (F+G)^{-1}(y)=]\delta^{2}/2,+\infty[.$$
Thus, $$d(x,(F+G)^{-1}(y))=d(-\delta/2,]\delta^{2}/2,+\infty[)=\vert -\delta/2-(\delta^{2}/2)\vert=\delta/2+\delta^{2}/2,$$ and, 
$$d(y,(F+G)(x))=d(-\delta^{2}/2,\{-1,0\})=\min\{1-\delta^{2}/2,\delta^{2}/2\}=\delta^{2}/2.$$ Consequently, by (\ref{HDHM}), one obtains that $\delta/2+\delta^{2}/2\le \tau \delta^{2}/2$. Since,  $1<1+\delta\le \tau \delta$,  this yields 
$\delta>\tau^{-1},$ which contradits the choice of $\delta.$ Hence, $F+G$ can not metrically regular around $(0,0)$. \\
Of course, $(F,G)$ is not sum-stable around $(0,0,0).$ Indeed, 
 take $0<\varepsilon<1$, then, for every $\delta>0$, 
consider  $x_{\delta}:=\delta/2\in ]-\delta,\delta[$ and $w_{\delta}:=\delta/2\in (F+G)(x_{\delta})\cap ]-\delta,\delta[=]-\delta/2,\delta[$. By taking $\varepsilon$ smaller if necessary, we can assume that $\delta>2\varepsilon.$ Then, for every $y_{\delta}\in F(x_{\delta})\cap (-\varepsilon,\varepsilon)=]-\varepsilon,\varepsilon[$ and, for every $z_{\delta}\in G(x_{\delta})\cap ]-\varepsilon,\varepsilon[=\{0\},$ one has $w_{\delta}=\delta/2>\varepsilon+0>y_{\delta}+z_{\delta}.$\qed 
\end{proof}
The following theorem establishes   metric regularity of  the multifunction   $\mathcal{E}_{(F,G)}$ as well as metric regularity of the sum mapping, of course, with the sum-stable assumption added.
% THEOREM
\begin{theorem} \label{Slope O} Let $X$ be a complete metric
space, let $Y$ be a Banach space and let $F, G:X\rightrightarrows Y$ be  closed multifunctions. Suppose that  $(\bar{x},\bar{k},\bar{y})\in X\times Y\times Y$  is  such that $\bar{y}\in F(\bar{x})+\bar{k}, \bar{k}\in G(\bar{x})$,  $F$ be metrically regular around $(\bar{x},\bar{y}-\bar{k})$ with modulus $\tau>0$ and $G$ is  pseudo-Lipschitz  around $(\bar{x},\bar{k})$ with modulus $\lambda>0$ with $\tau\lambda<1.$ Suppose  that the product space $X\times Y$  is  endowed with the metric defined by 
$$d((x,k),(u,z))=\max\{d(x,u),\Vert z-k\Vert/\lambda\}.$$ Then $\mathcal{E}_{(F,G)}$ is metrically regular around $(\bar{x},\bar{k},\bar{y})$ with modulus $(\tau^{-1}-\lambda)^{-1}.$

If  in addition  we suppose that the pair $(F,G)$  is  locally sum-stable around $(\bar{x},\bar{y}-\bar{k}, \bar{k})$,  then  $F+G$ is metrically regular around $(\bar{x},\bar{y})$ with modulus $(\tau^{-1}-\lambda)^{-1}.$
\end{theorem}
\begin{proof}
Since by assumption $G$ is  pseudo-Lipschitz   around $(\bar{x},\bar{k})$ with modulus $\lambda>0$, there exist $\delta_1,\delta_2>0$ such that
\begin{equation}\label{pseudo}
G(x_1)\cap B(\bar{k},\delta_1)\subset G(x_2)+\lambda \Vert x_1-x_2\Vert \bar{B}_{Y}, \quad \text{for all} \; x_1,x_2\in B(\bar{x},\delta_2).
\end{equation}
Then,  obviously, $G$ is lower semicontinuous at  all $(x,k)\in (B(\bar{x},\delta_2)\times B(\bar{k},\delta_1))\cap\gph G.$. Therefore, $\varphi_{\mathcal{E}}$ is given by the second equality in Lemma \ref{tron}.
Furthermore, since $F$ is metrically regular around $(\bar{x},\bar{y}-\bar{k})$ with modulus $\tau>0$, there exist $\delta_3,\delta_4>0$ and a real $\gamma>0$ such that 
\begin{equation}\label{Cond Slope}
|\nabla\varphi_F(\cdot,y)|(x)\geq \frac{1}{\tau}\quad\mbox{for all}\; (x,y)\in
B(\bar{x},\delta_3)\times B(\bar{y}-\bar{k},\delta_4)\quad\mbox{with}\; \varphi_F(x,y)\in ]0,\gamma[.
\end{equation}
So,  for any $ \varepsilon >0$,  there exists $u\in B(x,\delta_3), u\ne x$  such that 
$$\frac{\varphi_F(x,y)-\varphi_F(u,y)}{d(x,u)}>\frac{1}{\tau+\varepsilon/2}.$$
Taking $\delta_1,\delta_3$  smaller if neccesary, we can assume that $\delta_1<\delta_4,$ and $\delta_3<\delta_2.$  
Then, for every $(x,k,y)\in B(\bar{x},\min\{\delta_2,\delta_3\}/2)\times B(\bar{k},\delta_1)\times B(\bar{y},\delta_4-\delta_1)$ with $y-k\notin F(x), k\in G(x),$  any 
$\varepsilon>0 $ and  any  sequence $\{x_n\}_{n\in\N}\subseteq X$ converging to $x$, $\{k_n\}_{n\in\N}\subseteq X$ converging to $k$ with $k_n\in G(x_n),$ and
$$\lim_{n\to\infty} d(y-k_n,F(x_n))=\liminf_{u\to x}d(y-k,F(u)),$$ we deduce that 
\begin{equation}\label{phi}
\frac{\varphi_F(x,y-k)-\varphi_F(u,y-k)}{d(x,u)}>\frac{1}{\tau+\varepsilon/2}, (\text{since}\; y-k\in  B(\bar{y}-\bar{k},\delta_4)),
\end{equation}
 and 
$$\lim_{n\to\infty} d(y-k,F(x_n))=\lim_{n\to\infty} d(y-k_n,F(x_n))=\liminf_{u\to x}d(y-k,F(u))=\varphi_F(x,y-k).$$
On the other hand, by definition of the function $\varphi_{\mathcal{E}}$, there is a sequence $\{u_n\}_{n\in\N}\subseteq X$ converging to $u$ such that 
$$\lim_{n\to\infty} d(y-k,F(u_n))=\varphi_F(u,y-k).$$
Because $u\in B(x,\delta_3), x\in B(\bar{x},\min\{\delta_2,\delta_3\}/2),\{u_n\}_{n\in\N}\to u$, for $n$ large enough, one has that $u_n\in B(\bar{x},\delta_2).$
Similarly, since $k\in B(\bar{k},\delta_1)$ and $\{k_n\}_{n\in\N}\subseteq X$ converges to $k$,  for $n$ large enough, one has that $k_n\in B(\bar{k},\delta_1).$ \\
Therefore, by (\ref{pseudo}), and (\ref{phi}), there exists $z_n\in G(u_n)$ such that 
\begin{equation}\label{1}
\Vert z_n-k_n\Vert\le \lambda d(x_n,u_n).
\end{equation}
and
$$\lim_{n\to\infty}\frac{d(y-k,F(x_n))-d(y-k,F(u_n))}{d(x,u)}>\frac{1}{\tau+\varepsilon}.$$
Thus, noting  that $ u\not= x$, one has that
\begin{align*}
 &\limsup_{n\to\infty}\frac{d(y-k,F(x_n))-d(y-k,F(u_n))}{d(x_n,u_n)}\\
&=\lim_{n\to\infty}\frac{d(y-k,F(x_n))-d(y-k,F(u_n))}{d(x_n,u_n)}\\
&=\lim_{n\to\infty}\frac{d(y-k,F(x_n))-d(y-k,F(u_n))}{d(x,u)}\frac{d(x,u)}{d(x_n,u_n)}\\
&=\lim_{n\to\infty}\frac{d(y-k,F(x_n))-d(y-k,F(u_n))}{d(x,u)}\lim_{n\to\infty}\frac{d(x,u)}{d(x_n,u_n)}\\&\le \frac{1}{\tau+\varepsilon}.\\
\end{align*}
On the other hand,
\begin{equation}\label{3}
d(y-z_n,F(u_n))\le d(y-k_nF(u_n))+\Vert k_n-z_n\Vert.
\end{equation}
From relations (\ref{1}), (\ref{3}), we deduce that for any $(x,k,y)\in B(\bar{x},\min\{\delta_2,\delta_3\}/2)\times B(\bar{k},\delta_1)\times B(\bar{y},\delta_4-\delta_1)$ with $y-k\notin F(x),k\in G(x),$ and any $\varepsilon>0,$ any  sequence $\{x_n\}_{n\in\N}\subseteq X$ converging to $x$, $\{k_n\}_{n\in\N}\subseteq X$ converging to $k,$ there exists $\{(u_n,z_n)\}_{n\in\N}$ with 
$$\liminf_{n\to\infty} d((u_n,z_n),(x,k))=\liminf_{n\to\infty}\max\{d(u_n,x),\Vert z_n-k\Vert/\lambda\} \geq \liminf_{n\to\infty}d(u_n,x)>0,$$
(since $0<d(x,u)\le d(u_n,x)+d(u_n,u)$ and $u_n\to u$)\\
 such that 
\begin{align*}
&\limsup_{n\to\infty}\frac{d(y-k_n,F(x_n))-d(y-z_n,F(u_n))}{d((x_n,k_n),(u_n,z_n))}\\
&\geq \limsup_{n\to\infty}\frac{d(y-k_n,F(x_n))-d(y-k_n,F(u_n))-\Vert k_n-z_n\Vert}{d((x_n,k_n),(u_n,z_n))}\\
&=\limsup_{n\to\infty}\frac{d(y-k_n,F(x_n))-d(y-k_n,F(u_n))-\Vert k_n-z_n\Vert}{\max\{d(x_n,u_n),\Vert k_n-z_n\Vert /\lambda \}}\\
&\geq\limsup_{n\to\infty}\frac{d(y-k_n,F(x_n))-d(y-k_n,F(u_n))}{\max\{d(x_n,u_n),\Vert k_n-z_n\Vert /\lambda \}}-\lambda\\
&=\limsup_{n\to\infty}\frac{d(y-k_n,F(x_n))-d(y-k_n,F(u_n))}{d(x_n,u_n)}-\lambda>\frac{1}{\tau+\varepsilon}-\lambda,\\
\end{align*}

$(\text{since} \;\Vert z_n-k_n\Vert/\lambda \le d(x_n,u_n)).$\\
By Lemma \ref{strophi} ($(i)\Leftrightarrow (ii)$), one concludes that $\mathcal{E}_{(F,G)}$ is metrically regular around $(\bar{x},\bar{k},\bar{y})$ with modulus $(\tau^{-1}-\lambda)^{-1}.$  

If   the pair $(F,G)$ is locally sum-stable around $(\bar{x}, \bar{y}-\bar{k}, \bar{k}),$ then,  combining the hypothesis  Proposition \ref{me-sum} and Proposition \ref{Slope E}, we  complete the proof. \qed
\end{proof}
Combining Proposition \ref{Slope E} and Theorem \ref{Slope O}, we obtain  the following corollary, which is equivalent to the main result  (Theorem 3.3) in \cite{Dustru}, which is stated for the  difference of an open mapping and a  pseudo-Lipschitz  one.
% COROLLARY
\begin{corollary}
Let $X$ be a complete metric space, let $Y$ be a Banach space and let $F, G:X\rightrightarrows Y$ be  closed multifunctions. Suppose that  $(\bar{x},\bar{k},\bar{y})\in X\times Y\times Y$  is   such that $\bar{y}\in F(\bar{x})+\bar{k}, \bar{k}\in G(\bar{x})$ and $F$ is  metrically regular around $(\bar{x},\bar{y}-\bar{k})$ with modulus $\tau>0$ and $G$  is   pseudo-Lipschitz  around $(\bar{x},\bar{k})$ with modulus $\lambda>0$ with $\tau\lambda<1.$ Then, there exist a neighborhood  $\;\mathcal{U}\times \mathcal{V}\times \mathcal{W}$ of $(\bar{x},\bar{k},\bar{y}-\bar{k})$  and $\varepsilon,\tau>0$ such that,  for every $(x,k,z)\in \mathcal{U}\times \mathcal{V}\times \mathcal{W}, k\in G(x), z\in F(x),$ and 
 $\rho\in ]0,\varepsilon[,$
$$B(k+z,\rho\tau^{-1})\subset  (F+G)(B(x,\rho)).$$ 
\end{corollary}
\section{Metric Regularity of the Epigraphical Multifunction under Coderivative Conditions}
In this section, $X, Y$ are assumed to be Asplund spaces, i.e., Banach spaces for which each  separable subspace has a separable dual (in particular, any reflexive space is Asplund; see, e.g., \cite{BZ05, Borisbook1} for more details).    We recall some notation, terminology  and   definitions basically   standard and conventional   in the area
of variational analysis and generalized differentials (see \cite{BZ05, Borisbook1, Roc-Wet, Schiro, penot2012, LPX2011} and the references therein).
As usual, $\|\cdot\|$
stands for the norm on $X$ or $Y$, indifferently,  and $\la\cdot,\cdot\ra$
signifies for the canonical pairing between $X$ and its
topological dual $X^\star$ with the symbol $\st{w^\star}\to$ indicating
the convergence in the weak$^\star$ topology of $X^\star$ and the symbol
$\cl$ standing for the weak$^\star$ topological closure of a set. % 
Given  a set-valued mapping $F\colon X\rightrightarrows X^\star$ between $X$
and $X^\star$, recall that the symbol
\begin{equation}\label{jardin}
\Limsup_{x\rightarrow \ox} F(x):=\Big\{x^{\star}\in
X^{\star}\Big|\;\exists\,x_n\to\ox,\;\exists\,
x^{\star}_n\st{w^{\star}}\to x^{\star}\;\mbox{ with }\;x^{\star}_n\in F(x_n),\quad
n\in\N\Big\}
\end{equation}
stands for the {\em sequential Painlev\'{e}-Kuratowski outer/upper
limit} of $F$ as $x\to\ox$ with respect to the norm topology of
$X$ and the weak$^\star$ topology of $X^\star$.
Let us consider  $f: X\to \R\cup\{+\infty\}$ an  extended-real-valued lower semicontinuous function and $\bar{x}$ fixed in $X$.  The  notation  $x\st{f}{\to}\ox$ means  that 
with
$x\to\ox$ with $f(x)\to f(\ox)$. 
The Fr\'echet subdifferential  $\hat{\partial}f(\bar{x})$ of $f$ at $\bar{x}$ is given by the formula:  $$\hat{\partial}f(\bar{x})=\left\{x^{\star}\in X^{\star}: \liminf_{x\to \bar{x},\; x\ne \bar{x}}\frac{f(x)-f(\bar{x})-\langle x^{\star},x-\bar{x}\rangle}{\Vert x-\bar{x}\Vert}\geq 0\right\},$$
and $\hat{\partial}f(\bar{x})=\emptyset$ if $\bar{x}\notin\dom f.$

 The notation $ \partial f(\bar{x})$ is used  to denote the  limiting subdifferential of $f$ at $\bar{x}\in\dom f$. It is defined by
$$ \partial f(\bar{x}):=\Limsup_{x\overset {f} {\rightarrow}  \bar{x}}\hat{\partial}f(x).$$
 For a closed set $C\subset X$ and $\bar{x}\in C$, the Fr\'echet normal cone to $C$ at $\bar{x}$ is  denoted   $\hat{N}(\bar{x};C)$ and is defined as the    Fr\'echet subdifferential of indicator function $\delta_{C}$ of $C$  at $\bar{x},$ i.e., 
$$  \hat{N}(\bar{x};C):=\hat{\partial}\delta_{C}(\bar{x}),$$ where $\delta_{C} (x)=0\; \text{if}\; x\in C$, and $\delta_{C} (x)=+\infty\; \text{if}\; x\notin C.$

The limiting normal cone of $C$ at $\bar{x}$ is defined and denoted by 
$$N(\bar{x};C)=\partial \delta_{C}(\bar{x}).$$\
Let us   consider a closed multifunction  $F:X \rightrightarrows Y$  and $\bar{y}\in F(\bar{x}).$  The Fr\'echet coderivative of $F$ at $(\bar{x},\bar{y})$ is the mapping $\hat{D}^{\star}F(\bar{x},\bar{y}):Y^{\star} \rightrightarrows X^{\star}$ defined by
 $$x^{\star}\in \hat{D}^{\star}F(\bar{x},\bar{y})(y^{\star})\Leftrightarrow (x^{\star},-y^{\star})\in \hat{N}((\bar{x},\bar{y});\gph F),$$
 while the  Mordukhovich (limiting) coderivative of $F$ at $(\bar{x},\bar{y})$ is the mapping $D^{\star}F(\bar{x},\bar{y}):Y^{\star} \rightrightarrows X^{\star}$ defined by
 $$x^{\star}\in D^{\star}F(\bar{x},\bar{y})(y^{\star})\Leftrightarrow (x^{\star},-y^{\star})\in N((\bar{x},\bar{y});\gph F).$$
 Here, $\hat{N}((\bar{x},\bar{y});\gph F)$ and  $N((\bar{x},\bar{y});\gph F)$  are   the Fr\'echet and the  limiting normal cone to $\gph F$ at $(\bar{x},\bar{y}),$ respectively. 

To obtain a point-based condition for metric regularity of multifunctions  in infinite dimensional spaces, one often uses  the so-called   \textit{partial sequential normal compactness}   (PSNC) property.  \\
A  multifunction $F:X\rightrightarrows Y$  is \textit{partially sequentially normally compact}   at $(\bar{x},\bar{y})\in \gph F$, iff,  for any sequences  $\{(x_k , y_k , x_{k}^{\star} ,y_{k}^{\star})\}\in \gph F \times X^{\star} \times Y^{\star}$
satisfying $$(x_k , y_k) \to (\bar{x}, \bar{y}), x_{k}^{\star} \in \hat{D}^{\star}(x_k , y_k)(y_{k}^{\star}), x^{\star}_{k}\overset {w^{\star}} {\rightarrow} 0, \Vert y_{k}^{\star}\Vert\to 0,$$ 
one has $\Vert x^{\star}_{k}\Vert \to 0$ as $k\to \infty$.
%REMARK
\begin{remark}
Condition (PSNC) at $(\bar{x},\bar{y})\in \gph F$ is satisfied if $X$ is finite dimensional, or  $F$ is pseudo-Lipschitz around that point.
\end{remark} 
In the following, we need a result on {\it the metric inequality} (see, e.g., Ioffe \cite{ioffe01}, Huynh \&Th\'era \cite{RefNT1}). Let us recall that 
the sets  $\{\Omega_1, \Omega_2\}$  satisfy the metric inequality at $\bar{x}$ iff,  there are $\tau>0$ and $r>0$ such that $$d(x,\Omega_1\cap \Omega_2)\le \tau[d(x,  \Omega_1)+d(x,  \Omega_2)] \; \text{for all}\; x\in B(\bar{x},r).$$
\begin{definition}
We say that at  $\bar{x}$, property  $(\mathcal{H})$ is satisfied if 
\vskip 2mm\noindent
for any sequences $\{x_{ik}\}_{k\in\N} \subset  \Omega_{i} \ (i=1,2), \quad \{{x_{ik}}^{\ast}\}\in\hat{N}(x_{ik};\Omega_i)_{k\in\N}\  (i=1,2)$ such that  $\{x_{ik}\}_{k\in\N}\to \bar{x},$ and  $\|{x^{\star}_{1k}}+{x^{\star}_{2k}} \|_{k\in \N} \to 0$, then necessarily $\;\{x^{\star}_{1k}\} \to 0$ and   $\{x^{\star}_{2k}\} \to 0.$ 
\end{definition}
 Property $(\mathcal{H})$ was   called by 
A. Y. Kruger,   dual (or normal) uniform regularity (see, \cite{SVA} and \cite{TJM}  for a comparison between hypothesis   $(\mathcal{H})$ and the metric inequality.  One can also  note  that $(\mathcal{H})$ 
 is the Asplund space version of the Mordukhovich ``limiting qualification condition"  (cf. \cite[Definition 3.2 (ii)]{Borisbook1}). Although formally the last one is weaker, it is easy to show that in the Asplund space setting the two conditions  are equivalent.
%PROPOSITION
\begin{proposition}\label{Fuzzy}
 Let  $\{\Omega_1, \Omega_2\}$ be two closed subsets of $X$ and fix  $\bar{x}\in \Omega_1\cap \Omega_2.$ If we suppose that property   $(\mathcal{H})$  holds, then
the sets  $\{\Omega_1, \Omega_2\}$ satisfy the metric inequality at $\bar{x}.$ 
Under this assumption,  there is some  $r>0$ such that   for every $ \varepsilon>0$,  and $x\in B(\bar{x},r),$ there exist  $x_1, x_2\in B(x,\varepsilon)$ such that
\begin{equation}\label{metricin}
\hat{N}(x;\Omega_1\cap\Omega_2)\subset \hat{N}(x_1;\Omega_1)+\hat{N}(x_2;\Omega_2)+\varepsilon B_{X^\star}.
\end{equation}
\end{proposition}
Let us consider two multifunctions $F, G: X\rightrightarrows Y$. To these multifunctions, we associate the two sets  $$C_1:=\{(x,y,z)\in X\times Y\times Y: y\in G(x)\} \;\text{and}\; C_2:=\{(x,y,z)\in X\times Y\times Y: z\in F(x)\}.$$ 
% REMARK
\begin{remark}
Hypothesis $(\mathcal{H})$ can be restated  for the sets $\{C_1,C_2\}$ at $(\bar{x},\bar{y},\bar{z})\in C_1\cap C_2$  as follows:
\item{(i)} $(\mathcal{H})$: for any sequences $$\{(x_{k}, y_{k})\}_{k\in\N} \subset \gph G,\{(v_{k}, z_{k})\}_{k\in\N} \subset \gph F,$$
 $$x^{\star}_k\in \hat{D}^{\star}G(x_{k}, y_{k})(y^{\star}_k),u^{\star}_k\in \hat{D}^{\star}F(v_{k}, z_{k})(z^{\star}_k),$$ such that if
 $$(x_{k}, y_{k})\to (\bar{x},\bar{y}),$$
 $$ (v_{k}, z_{k})\to (\bar{x},\bar{z}),$$
 $$ \Vert x^{\star}_k+u^{\star}_k\Vert\to 0,$$
 $$y^{\star}_k\to 0,z^{\star}_k\to 0,$$ then $$x^{\star}_k\to 0, u^{\star}_k\to 0, \text{as} \;k\to 0.$$ It  holds whenever one of following conditions is fulfilled: 
   \item{(ii)} $F^{-1}$ or $G^{-1}$ is pseudo-Lipschitz  around $(\bar{z},\bar{x})$  or $(\bar{y},\bar{x}),$ respectively;
 \item{(iii)} either $F$  is PSNC at $(\bar{x},\bar{z})$ or $G$ is PSNC at $(\bar{x},\bar{y}),$  and $$D^{\star}F(\bar{x},\bar{z})(0)\cap -D^{\star}G(\bar{x},\bar{y})(0)=\{0\}.$$
\end{remark}
\begin{proof} Observe that,   if $F^{-1}$ or $G^{-1}$ is pseudo-Lipschitz  around $(\bar{z},\bar{x})$ and $(\bar{y},\bar{x}),$ respectively,  then assumption  $(\mathcal{H})$ always holds (see for instance  \cite{Borisbook1}). \\ We now assume that (iii) holds. Take  $$\{(x_{k}, y_{k})\}_{k\in\N} \subset \gph G,\{(v_{k}, z_{k})\}_{k\in\N} \subset \gph F,$$
 $$x^{\star}_k\in \hat{D}^{\star}G(x_{k}, y_{k})(y^{\star}_k),u^{\star}_k\in \hat{D}^{\star}F(v_{k}, z_{k})(z^{\star}_k),$$ such that 
 $$(x_{k}, y_{k})\to (\bar{x},\bar{y}),$$
 $$ (v_{k}, z_{k})\to (\bar{x},\bar{z}),$$
 $$ \Vert x^{\star}_k+u^{\star}_k\Vert\to 0,$$
 $$y^{\star}_k\to 0,z^{\star}_k\to 0.$$ 
 If  the sequences $\{ x^{\star}_{k}\}_{k\in\N}$, $\{u^{\star}_{k}\}_{k\in\N}$  are unbounded,  we can assume that 
$$\Vert x^{\star}_{k}\Vert\to \infty, \Vert u^{\star}_{k}\Vert\to \infty,$$ and $$\frac{x^{\star}_{k}}{\Vert x^{\star}_{k}\Vert}\overset {w^{\star}} {\rightarrow}  x^{\star}, \frac{u^{\star}_{k}}{\Vert u^{\star}_{k}\Vert}\overset {w^{\star}} {\rightarrow}  u^{\star}.$$  
Then, $$y^{\star}_{k}/\Vert x^{\star}_{k}\Vert\to 0\; \text{and} \; z^{\star}_{k}/\Vert u^{\star}_{k}\Vert\to 0.$$ 
Consequently, $$x^{\star}\in D^{\star}G(\bar{x},\bar{y})(0), u^{\star}\in D^{\star}F(\bar{x},\bar{z})(0).$$
On the other hand, $$u^{\star}+x^{\star}=0, \;(\text{since}\; \Vert x^{\star}_{k}+u^{\star}_{k}\Vert\to 0).$$ 
It follows that $$u^{\star}\in D^{\star}F(\bar{x},\bar{z})(0)\cap -D^{\star}G(\bar{x},\bar{y})(0).$$
Therefore, by  assumption, this yields $x^{\star}=u^{\star}=0.$ \\
Hence, $$\frac{x^{\star}_{k}}{\Vert x^{\star}_{k}\Vert}\to 0, \quad \text{or}\quad \frac{u^{\star}_{k}}{\Vert u^{\star}_{k}\Vert}\to 0,\; \text{(by PSNC property of $F$ or  $G$).} $$ This contradicts the fact that $\frac{x^{\star}_{k}}{\Vert x^{\star}_{k}\Vert}$, and $\frac{u^{\star}_{k}}{\Vert u^{\star}_{k}\Vert}$ are in the unit sphere $S_{Y^{\star}}$ of $Y^{\star}.$
So, the sequences $\{ x^{\star}_{k}\}_{k\in\N}$, $\{u^{\star}_{k}\}_{k\in\N}$  are bounded. Without  any loss of generality, we can assume that  $$x^{\star}_{k} \overset {w^{\star}} {\rightarrow} x^{\star} , u^{\star}_{k}\overset {w^{\star}} {\rightarrow} u^{\star}.$$
It follows that $$x^{\star}\in D^{\star}G(\bar{x},\bar{y})(0), u^{\star}\in D^{\star}F(\bar{x},\bar{z})(0).$$ Moreover, $$x^{\star}+u^{\star}=0.$$
Hence, $$u^{\star}\in D^{\star}F(\bar{x},\bar z)(0)\cap -D^{\star}G(\bar{x}, \bar{y})(0).$$
Therefore, by assumption, we obtain  $x^{\star}=u^{\star}= 0, $ and $x^{\star}_{k}\to 0,\;\text{or,}\; u^{\star}_{k}\to 0,$ (by PSNC property of $F$ or $G$). The proof is complete. \qed 
\end{proof}
The following lemma gives an estimation
for the strong slope of the function $\varphi_{\mathcal{E}}((x,k),y).$
\begin{lemma} \label{Estim Slope} Let $(\bar x,\bar y-\bar k,\bar k )\in X\times F(\bar x)\times G(\bar{x})$ be given. Assume that the sets $\{C_1,C_2\}$ defined, as above,   satisfy  hypothesis  $(\mathcal{H})$ at $(\bar x,\bar k,\bar y-\bar k)$. Then there exists $\rho>0$ such that,  for all $(x,k,y)\in B((\bar x,\bar k,\bar y),\rho)$ with
$y\notin F(x)+k, k\in G(x)$ as well as $d(y,F(x)+k)<\rho,$ one has
\begin{equation}
|\nabla \varphi_{\mathcal{E}}((\cdot,\cdot),y)|(x,k)\ge\lim_{\delta\downarrow
0} \left\{\inf\left\{\|x^{\star}\|:\;
\begin{array}{l}
(u,w)\in\gph F,(v,z)\in\gph G,u,v\in B(x,\delta),\\
u^{\star}\in \hat{D}^{\star}G(v,z)(y^{\star}), \|y^{\star}\|=1,z\in B(k,\delta)\\
x^{\star}\in \hat{D}^{\star}F(u,w)(y^{\star}+z^{\star})+u^{\star},z^{\star}\in\delta B_{Y^{\star}},\\
 |\|w+k-y\|-\varphi_{\mathcal{E}}((x,k),y)|<\delta,\\
|\langle y^{\star}+z^{\star},w+k-y\rangle-\|w+k-y\||<\delta
\end{array}\right\}
 \right\}.\notag
\end{equation}
\end{lemma} 
\begin{proof}
Obviously, if $(\mathcal{H})$ is satisfied at $(\bar x,\bar k,\bar y-\bar k)$ then it is also satisfied at all points $(u,v,w)\in X\times G(u)\times F(u)$ near $(\bar x,\bar k,\bar y-\bar k),$ say $(u,v,w)\in X\times G(u)\times F(v)\cap B_{X\times Y\times Y}((\bar x,\bar k,\bar y-\bar k),3\rho).$
%%%
 Let $(x,k,y)\in B_{X\times Y\times Y}((\bar x,\bar k,\bar y-\bar k),\rho)$ be such that
$y\notin F(x)+k, k\in G(x)$ and $d(y, F(x)+k)<\rho.$ Set 
 $|\nabla \varphi_{\mathcal{E}}((\cdot,\cdot),y)|(x,k):=m.$ 
By the
lower semicontinuity of $\varphi_{\mathcal{E}}$ (Note that    $\varphi_{\mathcal{E}}$ is given by the first equality in Lemma \ref{tron})  as well as the definition of the
strong slope, for each $\varepsilon\in]0,\varphi_{\mathcal{E}}((x,k),y)[,$ there is
$\eta\in (0,\varepsilon)$ with $4\eta+\varepsilon<\varphi_{\mathcal{E}}((x,k),y)$ and $1-(m+\varepsilon+3)\eta >0$  such
that $d(y,F(u)+l)\geq \varphi_{\mathcal{E}}((x,k),y)-\varepsilon,\; \text{for all} \;u\in
B(x,4\eta),$ $ l\in B(k,\eta)\cap G(u)$ and
$$m+\varepsilon\geq \frac{\varphi_{\mathcal{E}}((x,k),y)-\varphi_{\mathcal{E}}((z,k'),y)}{\max\{\|x-z\|,\|k-k'\| 
\}}\quad \mbox{for all}\; z\in  \bar{B}(x,\eta), k'\in \bar{B}(k,\eta)\cap G(x).$$
Consequently,
$$\varphi_{\mathcal{E}}((x,k),y)\leq \varphi_{\mathcal{E}}((z,k'),y)+(m+\varepsilon)\|z-x\|+(m+\varepsilon)\|k-k'\|\quad \mbox{for all}\; z\in \bar{B}(x,\eta), k'\in \bar{B}(k,\eta)\cap G(x).$$
By the definition of $\varphi_{\mathcal{E}},$ take $u\in B(x,\eta^2/4), \ v\in F(u), l\in B(k,\eta^2/8)\cap G(u) $ such that $$\|y-l-v\|\leq
\varphi_{\mathcal{E}}((x,k),y)+\eta^2/8.$$ By this way,
$$ \|y-k-v\|\leq
\varphi_{\mathcal{E}}((x,k),y)+\eta^2/4. $$
 Taking into account that   $\varphi_{\mathcal{E}}((z,k'),y)\le d(y,F(z)+k')$ with $k'\in G(z)$, then\\   $\varphi_{\mathcal{E}}((z,k'),y)\le \Vert y-k'-w\Vert$ with $w\in F(z)$ and $k'\in G(z)$. It follows that $$\varphi_{\mathcal{E}}((z,k'),y)\le \Vert y-k'-w\Vert+\delta_{C_2}(z,k',w)+\delta_{C_1}(z,k',w).$$
From the inequality, $$\|y-k-v\| \leq \varphi_{\mathcal{E}}((z,k'),y)+(m+\varepsilon)\|z-x\|+\eta^2/4,$$ we obtain that $$\|y-k-v\|\leq \|y-k'-w\|+\delta_{C_2}(z,k',w)+\delta_{C_1}(z,k',w)+(m+\varepsilon)\|z-u\|+(m+\varepsilon)\eta+\eta^2/4,$$ 
$ \text{for all} \;(z,w)\in \bar{ B}(x,\eta)\times Y, k'\in \bar{ B}(k,\eta).$ %
Applying the Ekeland variational principle \cite{RefE} to the function
$$(z,k',w)\mapsto \|y-k'-w\|+\delta_{C_2}(z,k',w)+\delta_{C_1}(z,k',w)+(m+\varepsilon)\|z-u\|$$
on $\bar{B}(x,\eta)\times  \bar{ B}(k,\eta)\times Y,$ we can select $(u_1,k_1,w_1)\in
(u,k,v)+\frac{\eta}{4} B_{X\times Y\times Y}$ with $(u_1,k_1,w_1)\in C_2\cap C_1$
such that \begin{equation}\label{estim 1}\|y-k_1-w_1\|\leq\|y-k-v\| (\leq
\varphi_{\mathcal{E}}((x,k),y)+\eta^2/4);\end{equation} and the function
$$(z,k',w)\mapsto\|y-k'-w\|+\delta_{C_2}(z,k',w)+\delta_{C_1}(z,k',w)+(m+\varepsilon)\|z-u\|+(m+\varepsilon+1)\eta\|(z,k',w)-(u_1,v_1,w_1)\|$$
attains a minimum on $\bar{ B}(x,\eta)\times\bar{ B}(k,\eta)\times Y$ at $(u_1,k_1,w_1).$ Hence, using the sum rule for Fr\'echet subdifferentials, we can find $$(u_2,k_2,w_2), (u_4,k_4,w_4)\in B_{X\times Y\times Y}((u_1,k_1,w_1),\eta);\;(u_3,k_3,w_3)\in
B_{X\times
Y\times Y}((u_1,k_1,w_1),\eta)\cap C_2\cap C_1;$$
such that 
\begin{align*}
(0,k_{2}^{\star},w_{2}^{\star})  &  \in\hat{\partial}\Vert y-\cdot-\cdot\Vert
(u_{2},k_{2},w_{2}),\\
(u_{3}^{\star},k_{3}^{\star},w_{3}^{\star})  &  \in \hat{\partial}%
(\delta_{C_2}(\cdot,\cdot,\cdot)+\delta_{C_1}(\cdot,\cdot,\cdot))(u_{3},k_{3},w_{3}),\\
(u_{4}^{\star},0,0)  &  \in\hat{\partial}((m+\varepsilon)\Vert\cdot-u\Vert
)(u_{4},k_{4},w_{4})
\end{align*}
and
\begin{equation}
(0,0,0)\in(0,k_{2}^{\star},w_{2}^{\star})+(u_{3}^{\star},k_{3}^{\star},w_{3}%
^{\star})+(u_{4}^{\star},0,0)+(m+\varepsilon+2)\eta\lbrack \bar B_{X^{\star}}\times
\bar B_{Y^{\star}}\times \bar B_{Y^{\star}}]. \label{fermat}%
\end{equation}
Note  that
\begin{align}
\Vert y-k_{2}-w_{2}\Vert &  \geq\Vert y-v-k\Vert-\Vert w_{2}-v\Vert-\Vert
k_{2}-k\Vert\label{ineg}\\
&  \geq\varphi_{\mathcal{E}}((x,k),y)-\varepsilon-(\Vert w_{2}%
-w_{1}\Vert+\Vert w_{1}-v\Vert)-(\Vert k_{2}-k_{1}\Vert+\Vert k-k_{1}%
\Vert)\nonumber\\
&  >\varphi_{\mathcal{E}}((x,k),y)-\varepsilon-2\eta-2\eta
=\varphi_{\mathcal{E}}((x,k),y)-\varepsilon-4\eta>0.\nonumber
\end{align}
Then, by \cite[Theorem 2.8.3]{Zal2002} (see, also  \cite [proof of Theorem
3.6]{DurStr1}), we know that 
\[
\hat{\partial}\Vert y-\cdot-\cdot\Vert(u_{2},k_{2},w_{2})=\{(0,y^{\star},y^{\star
}):y^{\star}\in S_{Y^{\star}},\langle y^{\star},w_{2}+k_{2}-y\rangle=\Vert
y-w_{2}-k_{2}\Vert\}.
\]
Hence,%
\begin{equation}\label{bis}
w_{2}^{\star}=k_{2}^{\star}\in S_{Y^{\star}}\text{ and }\langle w_{2}^{\star
},w_{2}+k_{2}-y\rangle=\Vert y-w_{2}-k_{2}\Vert.
\end{equation}
Now,  in order to have  $(u_3,k_3,w_3)\in B_{X\times Y\times Y}((\bar x,\bar k,\bar y-\bar k),3\rho),$  we take  $\eta$ smaller if necessary,  and,  by virtue of Proposition \ref{Fuzzy},   one has
\begin{equation}
(u_{3}^{\star},k_{3}^{\star},w_{3}^{\star}) \in \hat{N}((u_{5},k_{5},w_{5}); C_2)+
\hat{N}((u_{6},k_{6},w_{6}); C_1)+\eta\lbrack \bar B_{X^{\star}}\times
\bar B_{Y^{\star}}\times \bar B_{Y^{\star}}], \label{fuzzy}
\end{equation}
 where $(u_{5},k_{5},w_{5})\in C_2\cap B_{X\times
Y\times Y}((u_3,k_3,w_3),\eta), (u_{6},k_{6},w_{6})\in C_1\cap B_{X\times
Y\times Y}((u_3,k_3,w_3),\eta).$
From (\ref{fermat}) and (\ref{fuzzy}), one deduces that
$$
(0,0,0)\in(0,k_{2}^{\star},w_{2}^{\star})+ \hat{N}((u_{5},k_{5},w_{5}); C_2)+$$
$$\hat{N}((u_{6},k_{6},w_{6}); C_1)+(u_{4}^{\star},0,0)+(m+\varepsilon+3)\eta\lbrack \bar B_{X^{\star}}\times\bar B_{Y^{\star}}\times \bar B_{Y^{\star}}]. $$
Therefore, there exist $(u_{5}^{\star},k_{5}^{\star},w_{5}^{\star})\in \lbrack \bar B_{X^{\star}}\times\bar B_{Y^{\star}}\times \bar B_{Y^{\star}}], (u_{6}^{\star},k_{6}^{\star},0)\in \hat{N}((u_{6},k_{6},w_{6}); C_1)$, i.e., $u_{6}^{\star}\in \hat{D}^{\star}G(u_{6},k_{6})(-k_{6}^{\star})$ such that
$$(-u_{4}^{\star}-(m+\varepsilon+3)\eta u_{5}^{\star}-u_{6}^{\star},-k_{2}^{\star}-(m+\varepsilon+3)\eta k_{5}^{\star}-k_{6}^{\star}, -w_{2}^{\star}-(m+\varepsilon+3)\eta w_{5}^{\star})\in \hat{N}((u_{5},k_{5},w_{5}); C_2).$$
It follows that $$-k_{2}^{\star}-(m+\varepsilon+3)\eta k_{5}^{\star}-k_{6}^{\star}=0,$$ and $$(-u_{4}^{\star}-(m+\varepsilon+3)\eta u_{5}^{\star}-u_{6}^{\star}, -w_{2}^{\star}-(m+\varepsilon+3)\eta w_{5}^{\star})\in \hat{N}((u_{5},w_{5}); \gph F).$$
Consequently,
$$-k_{6}^{\star}=k_{2}^{\star}+(m+\varepsilon+3)\eta k_{5}^{\star} \; \text{and} \;(-u_{4}^{\star}-(m+\varepsilon+3)\eta u_{5}^{\star}-u_{6}^{\star})\in \hat{D}^{\star}F(u_{5},w_{5})(w_{2}^{\star}+(m+\varepsilon+3)\eta w_{5}^{\star}).$$
Remark that $\|k_{6}^{\star} \|= \|-k_{2}^{\star}-(m+\varepsilon+3)\eta k_{5}^{\star} \|\geq 1-(m+\varepsilon+3)\eta>0.$ \\\
Hence, setting
 \begin{align*}
&y^{\star}:=(k_{2}^{\star}+(m+\varepsilon+3)\eta k_{5}^{\star})/\|k_{2}^{\star}+(m+\varepsilon+3)\eta k_{5}^{\star}\|,\\
&z^{\star}:=(w_{5}^{\star}-k_{5}^{\star})(m+\varepsilon+3)\eta /\|k_{2}^{\star}+(m+\varepsilon+3)\eta k_{5}^{\star}\|,\\
&x_{1}^{\star}:=u_{6}^{\star}/\|k_{2}^{\star}+(m+\varepsilon+3)\eta k_{5}^{\star}\|,\\
&x_{2}^{\star}:=(-u_{4}^{\star}-(m+\varepsilon+3)\eta u_{5}^{\star})/\|k_{2}^{\star}+(m+\varepsilon+3)\eta k_{5}^{\star}\|,\\
 \end{align*}
one obtains that 
\begin{equation}
x_{1}^{\star}\in \hat{D}^{\star}G(u_{6},k_{6})(y^{\star}) \;\text{and}\; (x_{2}^{\star}-x_{1}^{\star})\in \hat{D}^{\star}F(u_{5},w_{5})(y^{\star}+z^{\star}),\label{coderivative}
\end{equation}
where
\begin{equation}
\|y^{\star}\|=1,\\
\|z^{\star}\|\le \frac{2(m+\varepsilon+3)\eta}{1-(m+\varepsilon+3)\eta}:=\delta,\\
\|x_{2}^{\star} \|\le \frac{m+\varepsilon+(m+\varepsilon+3)\eta}{1-(m+\varepsilon+3)\eta}.
\end{equation}
 On the other hand,  since $k_1\in B(k,\eta), w_5\in B(w_1,2\eta),$  according to relation (\ref{estim 1}) one has 
 \begin{equation}\label{varphi}
 \end{equation}
 $$
\begin{array}{lll}
\varphi_{\mathcal{E}}((x,k),y)-\varepsilon-3\eta\\
&\le  \|y-k_1-w_1 \|-\|w_5-w_1 \| - \|k_1-k \|
&\le \|y-k-w_{5} \|\\& \le \|y-k_1-w_1 \| + \|w_5-w_1 \| + \|k_1-k \| \\
&\le \|y-k-v \|+3\eta \\
&\le \varphi_{\mathcal{E}}((x,k),y)+\eta^{2}/4+3\eta.
\end{array}
 $$
Consequently,
\begin{equation}\label{rela-bis} |\|y-k-w_{5} \|-\varphi_{\mathcal{E}}((x,k),y)| \le 3\eta+\varepsilon.
\end{equation}
On the other hand, one has
\begin{equation}\label{Estim 5} 
\langle y^{\star}+z^{\star},y-k-w_5\rangle-\|y-k-w_{5} \|
\le \delta \|y-k-w_{5} \|;\end{equation} 
and, by $k_2\in B(k,\eta);$ $w_2,w_5\in B(w_1,2\eta),$ from (\ref{bis}), one has the following estimates 
\begin{equation}\label{rela2-bis}
\begin{array}{ll}
&\langle y^{\star}+z^{\star},{k}+w_5-y\rangle\\
&=\displaystyle\frac{\langle w_{2}^{\star}+(m+\varepsilon+3)\eta w_{5}^{\star},k+w_5-y\rangle}{\|k_{2}^{\star}+(m+\varepsilon+3)\eta k_{5}^{\star}\|}\\
&=\displaystyle\frac{\langle w^{\star}_2,w_2+k_2-y\rangle+\langle
w^{\star}_2,w_5-w_2\rangle+\langle
w^{\star}_2,{k}-k_2\rangle+(m+\varepsilon+3)\eta\langle
w^{\star}_5,w_5+{k}-y\rangle}{\|k_{2}^{\star}+(m+\varepsilon+3)\eta k_{5}^{\star}\|}\\
&\geq\displaystyle\frac{\|w_2+k_2-y\|-3\eta-2\eta-(m+\varepsilon+3)\eta\| w_5+{k}-y\|}{(1+(m+\varepsilon+3)\eta)},\\
&\geq\displaystyle\frac{\|w_5+k-y\|(1-(m+\varepsilon+3)\eta)-8\eta}{(1+(m+\varepsilon+3)\eta)}
 \end{array}
\end{equation}
As $\varepsilon,\eta>0$ are arbitrary small, by combining relations
(\ref{coderivative})-(\ref{rela2-bis}), we complete the proof.
\end{proof} 
 %THEOREM
\begin{theorem} \label{Slope1} Let $X, Y$ be Asplund
spaces,  and let $F, G:X\rightrightarrows Y$ be  closed multifunctions. Suppose that  $(\bar{x},\bar{k},\bar{y})\in X\times Y\times Y$  be  such that $\bar{y}\in F(\bar{x})+\bar{k}, \bar{k}\in G(\bar{x})$ and  the sets $\{C_1,C_2\}$ satisfy the hypothesis  $(\mathcal{H})$ at $(\bar{x},\bar{k},\bar{y}-\bar{k}).$ Let  $m>0$. If there exist a neighborhood $\mathcal{U}\times \mathcal{V}\times \mathcal{W}$ of $(\bar{x},\bar{k},\bar{y})$ and $\gamma>0$ such that,  for each $(x,y,k)\in  \mathcal{U}\times \mathcal{V} \times \mathcal{W}$ with $y\notin F(x)+k, k\in G(x),$
$$m\le \lim_{\delta\downarrow
0} \left\{\inf\left\{\|x^{\star}\|:\;
\begin{array}{l}
(u,w)\in\gph F,(v,z)\in\gph G, u,v\in B(x,\delta),\\
u^{\star}\in \hat{D}^{\star}G(v,z)(y^{\star}), \|y^{\star}\|=1,z\in B(k,\delta),\\
x^{\star}\in \hat{D}^{\star}F(u,w)(y^{\star}+z^{\star})+u^{\star},z^{\star}\in\delta B_{Y^{\star},}\\
%d(y,F(u)+k)\le\gamma+\delta,\\
 \|w+k-y\|\le \gamma+\delta,\\
|\langle y^{\star}+z^{\star},w+k-y\rangle-\|w+k-y\||<\delta
\end{array}\right\} \right\},$$
then there exists a neighborhood $\mathcal{U}_1\times \mathcal{V}_1\times \mathcal{W}_1$ of $(\bar x,\bar k,\bar y)$ such that
$$ md((x,k),\S_{\mathcal{E}_{(F,G)}}(y))\leq \varphi_{\mathcal{E}}((x,k),y)\quad \mbox{for all}\quad (x,k,y)\in \mathcal{U}_1\times \mathcal{V}_1\times \mathcal{W}_1.$$
\end{theorem}
This theorem implies  the following result:
 %THEOREM
\begin{theorem} \label{Limtheo} Let $X, Y$ be Asplund
spaces,  and let $F, G:X\rightrightarrows Y$ be  closed multifunctions, and let  $(\bar{x},\bar{k},\bar{y})\in X\times Y\times Y$ be such that $\bar{y}\in F(\bar{x})+\bar{k}, \bar{k}\in G(\bar{x})$. Let  $m>0$. If  the sets $\{C_1,C_2\}$ satisfy the hypothesis  $(\mathcal{H})$ at $(\bar{x},\bar{k},\bar{y}-\bar{k})$ and
\begin{equation}
m<\liminf_{{(x_1,w)\overset {F} {\rightarrow}(\bar{x},\bar{y}-\bar{k})},{(x_2,z)\overset {G} {\rightarrow}(\bar{x},\bar{k})}, \delta \downarrow 0}   \left\{\|x^{\star}\|:\;
\begin{array}{l}
x^{\star}\in \hat{D}^{\star}F(x_1,w)(y^{\star}+\delta B_{Y^\star})+u^{\star}\\
u^{\star}\in \hat{D}^{\star}G(x_2,z)(y^{\star}), \|y^{\star}\|=1,\\
\end{array}\right\} 
\end{equation}
where the notations ${(x_1,w)\overset {F} {\rightarrow}(\bar{x},\bar{y}-\bar{k})},$  ${(x_2,z)\overset {G} {\rightarrow}(\bar{x},\bar{k})}$ mean that $$(x_1,w)\to(\bar{x},\bar{y}-\bar{k}),(x_2,z)\to(\bar{x},\bar{k}) \;\text{and}\; (x_1,w)\in\gph F, (x_2,z)\in\gph G,$$ then there exists a neighborhood $\mathcal{U}_1\times \mathcal{V}_1\times \mathcal{W}_1$ of $(\bar x,\bar k,\bar y)$ such that
$$ md((x,k),\S_{\mathcal{E}_{(F,G)}}(y))\leq \varphi_{\mathcal{E}}((x,k),y)\quad \mbox{for all}\quad (x,k,y)\in U_1\times V_1\times W_1.$$
\end{theorem}
The next result gives a point-based condition for metric regularity of the epigraphical multifunction.
%  %THEOREM
 \begin{theorem} \label{basedpoint} Let $X, Y$ be Asplund
spaces,  and let $F, G:X\rightrightarrows Y$ be  closed multifunctions, and let  $(\bar{x},\bar{k},\bar{y})\in X\times Y\times Y$ be such that $\bar{y}\in F(\bar{x})+\bar{k}, \bar{k}\in G(\bar{x})$. Suppose that   
\item{(i)} $F$ or $G$ is PSNC at $(\bar{x},\bar{y}-\bar{k})$ and $(\bar{x},\bar{k}),$ respectively;
\item{(ii)} $D^{\star}F(\bar{x},\bar{y}-\bar{k})(0)\cap -D^{\star}G(\bar{x},\bar{k})(0)=\{0\};$
\item{(iii)} for any $u^{\star}_{n}\in \hat{D}^{\star}F(x_{n},y_{n}-k_{n})(y^{\star}_{n}+(1/n) B_{Y^{\star}}), v^{\star}_{n}\in \hat{D}^{\star}G(x_n,k_n)(y^{\star}_{n})$ such that $$\Vert u^{\star}_{n}+v^{\star}_{n}\Vert\to 0, y^{\star}_{n}\overset {w^{\star}} {\rightarrow} 0\; \text{it follows that } \; y^{\star}_{n}\to 0;$$
%\item{(i)} $F^{-1}$, or $G^{-1}$  is $(PSNC)$ at $(\bar{x},\bar{y}-\bar{k})$ and $(\bar{x},\bar{k})$, respectively; 
Under the condition that 
\begin{equation}\label{star?}\quad\quad \text{Ker}\big(D^{\star}F(\bar{x},\bar{y}-\bar{k})+D^{\star}G(\bar{x},\bar{k})\big)=\{0\},\end{equation}
 the multifunction $ \mathcal{E}_{(F,G)}$ is metrically regular around $(\bar{x},\bar{k},\bar{y}).$
\end{theorem} 
\begin{proof}
 We prove the result  by contradiction. Suppose  that $ \mathcal{E}_{(F,G)}$ fails to be metrically regular around $(\bar{x},\bar{k},\bar{y}).$ Then, by Theorem \ref{Limtheo}, there exist sequences $${(x_n,y_n-k_n)\overset {F} {\rightarrow}(\bar{x},\bar{y}-\bar{k})}, {(x_n,k_n)\overset {G} {\rightarrow}(\bar{x},\bar{k})}, (x^{\star}_{n}, u^{\star}_{n}, y^{\star}_{n},z^{\star}_{n})\in X^{\star}\times X^{\star}\times Y^{\star}\times Y^{\star},$$ with 
$$
\begin{array}{ll}
x^{\star}_{n}\in \hat{D}^{\star}F(x_{n},y_{n}-k_{n})(y^{\star}_{n}+z^{\star}_{n})+u^{\star}_{n},\\
u^{\star}_{n}\in \hat{D}^{\star}G(x_n,k_n)(y^{\star}_{n}), \\
y^{\star}_{n}\in S_{Y^{\star}}, z^{\star}_{n}\in (1/n )B_{Y^\star},\\
\end{array}
$$
and $$
\begin{array}{ll}
x^{\star}_{n}\to 0.
\end{array}
$$
Then there is $v^{\star}_{n}\in \hat{D}^{\star}F(x_{n},y_{n}-k_{n})(y^{\star}_{n}+z^{\star}_{n})$ such that $x^{\star}_{n}=u^{\star}_{n}+v^{\star}_{n}.$\\
Since $Y$ is an Asplund space, we can assume that $y^{\star}_{n}\overset {w^{\star}} {\rightarrow} y^{\star}$ $\in Y^{\star}.$ 

We consider the following cases:\\
Case 1.  The sequences $\{ u^{\star}_{n}\}_{n\in\N}$, $\{v^{\star}_{n}\}_{n\in\N}$  are unbounded.  We can assume that 
$$\Vert u^{\star}_{n}\Vert\to \infty, \Vert v^{\star}_{n}\Vert\to \infty,$$ and $$\frac{u^{\star}_{n}}{\Vert u^{\star}_{n}\Vert}\overset {w^{\star}} {\rightarrow}  u^{\star}, \frac{v^{\star}_{n}}{\Vert v^{\star}_{n}\Vert}\overset {w^{\star}} {\rightarrow}  v^{\star}.$$  
Then, $$y^{\star}_{n}/\Vert u^{\star}_{n}\Vert\to 0\; \text{and} \; (y^{\star}_{n}+z^{\star}_{n})/\Vert v^{\star}_{n}\Vert\to 0.$$ 
Consequently, $$u^{\star}\in D^{\star}G(\bar{x},\bar{k})(0), v^{\star}\in D^{\star}F(\bar{x},\bar{y}-\bar{k})(0).$$
On the other hand, $$u^{\star}+v^{\star}=0, (\text{since} \Vert u^{\star}_{n}+v^{\star}_{n}\Vert\to 0).$$ 
It follows that $$v^{\star}\in D^{\star}F(\bar{x},\bar{y}-\bar{k})(0)\cap -D^{\star}G(\bar{x},\bar{k})(0).$$
Therefore, by (ii), we have that $u^{\star}=v^{\star}=0.$ \\
So, $$\frac{u^{\star}_{n}}{\Vert u^{\star}_{n}\Vert}\to 0,\quad \text{or}\quad\frac{v^{\star}_{n}}{\Vert v^{\star}_{n}\Vert}\to 0,\; \text{(by PSNC property of F or G)}.$$ This contradicts 
the fact that  $\frac{u^{\star}_{n}}{\Vert u^{\star}_{n}\Vert}$ and $\frac{v^{\star}_{n}}{\Vert v^{\star}_{n}\Vert}$ belong to the unit sphere $S_{Y^{\star}}$ of $Y^{\star}.$ 

Case 2.   The sequences $\{ u^{\star}_{n}\}_{n\in\N}$, $\{v^{\star}_{n}\}_{n\in\N}$  are bounded.  Assume that $u^{\star}_{n} \overset {w^{\star}} {\rightarrow} u^{\star}$ , $v^{\star}_{n}\overset {w^{\star}} {\rightarrow} v^{\star}.$\\
It follows that $$u^{\star}\in D^{\star}G(\bar{x},\bar{k})(y^{\star}), v^{\star}\in D^{\star}F(\bar{x},\bar{y}-\bar{k})(y^{\star}).$$ Moreover, $$u^{\star}+v^{\star}=0.$$
So, $$0\in [D^{\star}G(\bar{x},\bar{k})(y^{\star})+D^{\star}F(\bar{x},\bar{y}-\bar{k})(y^{\star})]=[D^{\star}G(\bar{x},\bar{k})+D^{\star}F(\bar{x},\bar{y}-\bar{k})](y^{\star}),$$ which means that $$y^{\star}\in Ker[D^{\star}G(\bar{x},\bar{k})+D^{\star}F(\bar{x},\bar{y}-\bar{k})].$$
By $(\star)$, one has that $y^{\star}=0.$ \\
Now, by assumption, one gets $y^{\star}_{n}\to 0$ which contradicts  to $\Vert y^{\star}_{n}\Vert=1.$ \qed 
\end{proof}
%REMARK
\begin{remark}
If $X, Y$ are finite dimensional spaces,  then conditions  (i), (iii) hold true automatically, while  condition (ii)  holds if $F$ or $G$ is pseudo-Lipschitz  at $(\bar{x},\bar{y}-\bar{k})$ or  $(\bar{x},\bar{k}),$ respectively.
\end{remark}
\bigskip
\section{Applications to  Variational Systems}
In this section, we use the  above  results to study some properties  of variational systems of the form
\begin{equation}\label{varsys}
0\in F(x)+G(x,p),
\end{equation}
where  $X$ is   a complete metric space, $Y$ is a Banach space,   $P$ is  a topological  space considered as a parameter space, $F:X\rightrightarrows Y, G:X\times P\rightrightarrows Y$ are given multifunctions.  
The solution set  of (\ref{varsys}) is defined by 
\begin{equation}\label{solu}
{\bf{S}}_{(F+G)}(p):=\{x\in X: 0\in F(x)+G(x,p)\},
\end{equation}
and we denote $$\S_{(F+G)}(y,p):=\{x\in X: y\in F(x)+G(x,p)\}. $$
For every $(y,p)\in Y\times P,$
$$\S_{\mathcal{E}_{(F,G)}}(y,p)=\{(x,k)\in X\times Y: y\in F(x)+k, k\in G(x,p)\},$$and, for every $p\in P,$
$${\bf{S}}_{\mathcal{E}_{(F,G)}}(p)=\{(x,k)\in X\times Y: 0\in F(x)+k, k\in G(x,p)\}.$$ \\ 
We say that  the multifunction ${\bf{S}}_{(F+G)}$ is  Robinson metrically regular   (see  \cite{Rob1, Rob2})  around $(\bar{x},\bar{p})$ with modulus $\tau$,  iff there exist neighborhoods $\mathcal{U},\mathcal{V}$  of $\bar{x}, \bar{p},$ respectively, such that $$d(x,{\bf{S}}_{(F+G)}(p))\le \tau d(0,F(x)+G(x,p)),\;\text{for all} \;(x,p)\in\mathcal{U}\times \mathcal{V}.$$
We also recall that  the multifunction $G:X\times P\rightrightarrows Y$ is said to be pseudo-Lipschitz  around $(\bar{x},\bar{p},\bar{y})$ with $\bar{y}\in G(\bar{x},\bar{p})$ with respect to $x$, uniformly in $p$ with constant $\kappa >0$ iff there is a neighborhood  $\mathcal{U}\times \mathcal{V}\times \mathcal{W}$ of $(\bar{x},\bar{p},\bar{y})$ such that 
$$G(x,p)\cap \mathcal{W}\subset G(u,p)+\kappa d(x,u) \bar{B}_{Y} \; \text{for all} \; x,u\in \mathcal{U}, \text{and for all}\; p\in \mathcal{V}.$$
The lower semicontinuous envelope $(x,p,k,y)\mapsto \varphi_{p,\mathcal{E}}((x,k),y)$  of the distance function\\ $d(y,\mathcal{E}_{(F,G)}((x,p),k))$ is defined by,  for each $(x,p,k,y)\in X\times P\times Y\times Y$
$$\varphi_{p,\mathcal{E}}((x,k),y):=\liminf_{(u,v,w)\to (x,k,y)}d(w,\mathcal{E}_{(F,G)}((u,p),v))$$
\[
=\left\{
\begin{array}
[c]{ll}%
\liminf\limits_{ (u,v)\rightarrow (x,k), v\in G(u,p)}d(y,F(u)+k),  & \text{if $k\in G(x,p)$}\\
+\infty,  & \text{otherwise.}%
\end{array}
\right.
\]
%\end{lemma}
\begin{lemma}\label{(F,G)}
Let $X$ be a complete metric space and $Y$ be a Banach space and let
 $P$ be a topological  space. Suppose that the set-valued mappings
$F:X\rightrightarrows Y,$ $G:X\times P\rightrightarrows Y$ satisfy the following conditions
for some $(\bar{x},\bar{k},\bar{p})\in X\times Y\times P$:

\item{(a)} $(\bar {x},\bar{k})\in {\bf{S}}_{\mathcal{E}_{(F,G)}}(\bar{p});$

\item{(b)} the set-valued mapping  $p \rightrightarrows G(\bar {x},p)$ is  lower
semicontinuous at $\bar{p};$

\item{(c)} the set-valued mapping $F$ is a closed multifunction, and for any $p$ near $\bar{p},$ the set-valued mapping $x \rightrightarrows
G(x,p)$ is a closed multifunction.

Then 
\item{(i)}  for ever $p$ near $\bar{p}$, the epigraphical multifunction $\mathcal{E}_{(F,G)}$ has closed graph, and, $\mathcal{E}_{(F,G)}((\bar{x},\cdot),\bar{k})$  is lower semicontinuous at $\bar{p};$
\item{(ii)} the function $p\mapsto\varphi_{p,\mathcal{E}}((\bar{x},\bar{k}),0)$ is upper semicontinuous at $\bar{p};$
\item{(iii)} for each $(y,p)\in Y\times P;$
$$\{(x,k)\in X\times Y: \varphi_{p,\mathcal{E}}((x,k),y)=0\}=\S_{\mathcal{E}_{(F,G)}}(y,p).$$
\end{lemma}
\begin{proof}
 We only note that,  if the multifunction $p \rightrightarrows G(\bar {x},p)$ is  lower
semicontinuous at $\bar{p},$ then so is the mapping $\mathcal{E}_{(F,G)}((\bar{x},\cdot),\bar{k})$.\qed 
\end{proof}
By using the strong slope of the lower  semicontinuous envelope $\varphi_{p,\mathcal{E}}$, one has the following result.
 %THEOREM
\begin{theorem}\label{paravasys}
Let $X$ be a complete metric space, $Y$ be a Banach space and let $P$ be a topological space. Suppose that the set-valued mappings
$F:X\rightrightarrows Y, G:X\times P\rightrightarrows Y$ satisfy  conditions $(a), (b), (c)$ from Lemma \ref{(F,G)} around $(\bar{x},\bar{k},\bar{p})\in X\times Y\times P$. If there exist a neighborhood $\mathcal{T}_1\times \mathcal{U}_1\times \mathcal{V}_1\times \mathcal{W}_1$  of $(\bar{x},\bar{p},\bar{k},0)$ and  reals  $m,\gamma>0$  such that $\vert\nabla
\varphi_{p,\mathcal{E}}((\cdot,\cdot),y)\vert (x,k)\geq m$ for all $(x,p,k,y)\in \mathcal{T}_1\times \mathcal{U}_1\times \mathcal{V}_1\times \mathcal{W}_1$ with
$\varphi_{p,\mathcal{E}}((x,k),y)\in ]0,\gamma[,$ then 
there exists a neighborhood $\mathcal{T}\times \mathcal{U}\times \mathcal{V}\times \mathcal{W}$ of $(\bar{x},\bar{p},\bar{k},0)$ such that
$$md((x,k), \S_{\mathcal{E}_{(F,G)}}(y,p))\le \varphi_{p,\mathcal{E}}((x,k),y),$$
for all $(x,p,k,y)\in\mathcal{T}\times \mathcal{U}\times \mathcal{V}\times \mathcal{W}$.\\
\end{theorem}
\textit{Proof.} Applying Theorem \ref{Char2} and Lemma \ref{(F,G)} for the mapping $\mathcal{E}_{(F,G)}(\cdot,\cdot),$ one obtains the proof. 
%PROPOSITION
\begin{proposition}\label{papro} Let $X$ be a complete metric space and $Y$ be a Banach space and let
 $P$ be a topological space. Suppose that the set-valued mappings
$F:X\rightrightarrows Y, G:X\times P\rightrightarrows Y$ satisfy  conditions $(a), (b), (c)$ from Lemma \ref{(F,G)} around $(\bar{x},\bar{k},\bar{p})\in X\times Y\times P$.
 If there exist a neighborhood
  $\mathcal{T}\times \mathcal{U}\times \mathcal{V}\times \mathcal{W} \subset X\times P\times Y\times Y$ 
 of $(\bar{x},\bar{p},\bar{k},0)$ and $m>0$ such that
$$md((x,k), \S_{\mathcal{E}_{(F,G)}}(y,p))\le\varphi_{p,\mathcal{E}}((x,k),y) \quad \mbox{for all}\quad(x,p,k,y)\in   \mathcal{T}\times \mathcal{U}\times \mathcal{V}\times \mathcal{W} $$
 then
there exists   $\theta>0$ such that
$$md(x,\S_{(F+G)}(y,p))\leq d(y,F(x)+G(x,p)\cap B(\bar{k},\theta))\quad \mbox{for all}\quad (x,p,y)\in \mathcal{T}\times \mathcal{U}\times \mathcal{W} .$$
Therefore, 
$$md(x,{\bf{S}}_{(F+G)}(p))\leq d(0,F(x)+G(x,p)\cap B(\bar{k},\theta))\quad \mbox{for all}\quad (x,p)\in \mathcal{T}\times \mathcal{U}.$$
\end{proposition}
\begin{proof}
By the    hypothesis, there exist a neighborhood
  $\mathcal{T}\times \mathcal{U}\times \mathcal{V}\times \mathcal{W} \subset X\times P\times Y\times Y$ of $(\bar{x},\bar{p},\bar{k},0)$ and $m>0$ such that,
  for every $(x,p,k,y)\in \mathcal{T}\times \mathcal{U}\times \mathcal{V}\times \mathcal{W} ,$ it holds 
$$md((x,k), \S_{\mathcal{E}_{(F,G)}}(y,p))\le \varphi_{p,\mathcal{E}}((x,k),y).$$
Here, we can assume $\mathcal{V}=B(\bar{k},\theta)$, with certain positive $\theta$.
Then, for every small $\varepsilon>0$ and for  every $(x,p,k,y)\in \mathcal{T}\times \mathcal{U}\times [B(\bar{k},\theta)\cap G(x,p)]\times \mathcal{W},$ there is $(u,z)\in \S_{\mathcal{E}_{(F,G)}}(y,p),$ i.e., $y\in F(u)+z, z\in G(u,p)$
such that $$md(u,x)\le m\max\{d(u,x), \Vert z-k\Vert\}<(1+\varepsilon)d(y,F(x)+k).$$
Noting  that $u\in (F+G)^{-1}(y),$ we obtain that $$md(x,(F+G)^{-1}(y))<(1+\varepsilon)d(y,F(x)+k).$$
Thus, $$md(x,(F+G)^{-1}(y))\le (1+\varepsilon)d(y,F(x)+G(x,p)\cap B(\bar{k},\theta)),$$
or,  $$md(x,\S_{(F+G)}(y,p))\le (1+\varepsilon)d(y,F(x)+G(x,p)\cap B(\bar{k},\theta)).$$ \\
Since this inequality does not depend on arbitrarily small $\varepsilon>0$, we obtain that  
$$md(x,\S_{(F+G)}(y,p))\le d(y,F(x)+G(x,p)\cap B(\bar{k},\theta))$$ for all $(x,p,y)\in \mathcal{T}\times \mathcal{U}\times \mathcal{W}.$\\
Taking  $\bar y=0$ and $y=\bar y$, we obtain  the second conclusion of the Theorem. The proof is complete.\qed 
\end{proof}
 
In the sequel,  we use for the parametrized case  the concept of  locally sum-stability,   which  was considered in the previous section.
\begin{definition}
Let $F:X\rightrightarrows Y, G:X\times P\rightrightarrows Y$ be two multifunctions and $(\bar{x},\bar{p},\bar{y}, \bar{z})\in X\times P\times Y\times Y$  be such that $\bar{y}\in F(\bar{x}),\bar{z}\in G(\bar{x},\bar{p}).$ We say that the pair $(F,G)$ is locally sum-stable around $(\bar{x},\bar{p},\bar{y}, \bar{z})$ iff,  for every $\varepsilon>0, $ there exists $\delta>0$ and a neighborhood $W$ of $\bar p$ such that,  for every $(x,p)\in B(\bar{x},\delta)\times W$  and every $w\in (F+G)(x)\cap  B(\bar{y}+\bar{z},\delta),$ there are $y\in F(x)\cap  B(\bar{y},\varepsilon)$ and $z\in G(x)\cap  B(\bar{z},\varepsilon)$ such that $w=y+z.$
\end{definition}
A following simple case which ensures  the locally sum-stability of  the pair $(F,G),$ is analogous to Proposition \ref{usc-sum}.
%PROPOSITION
\begin{proposition}\label{usc-sum-para}
Let $F:X\rightrightarrows Y, G:X\times P\rightrightarrows Y$ be two multifunctions and $(\bar{x},\bar{p},\bar{y}, \bar{z})\in X\times P\times Y\times Y$ such that $\bar{y}\in F(\bar{x}),\bar{z}\in G(\bar{x},\bar{p}).$ If $G(\bar{x},\bar{p})=\{\bar{z}\}$ and $G$ is upper semicontinuous at $(\bar{x},\bar{p}),$ then the pair $(F,G)$ is locally sum-stable around $(\bar{x},\bar{p},\bar{y}, \bar{z}).$ 
\end{proposition}
%PROPOSITION
\begin{proposition}\label{@} Let $X$ be a complete metric space, $Y$ be a Banach space and let
 $P$ be a topological pace. Suppose that the set-valued mappings
$F:X\rightrightarrows Y, G:X\times P\rightrightarrows Y$ satisfy  conditions $(a), (b), (c)$ from  Lemma \ref{(F,G)} around $(\bar{x},\bar{k},\bar{p})\in X\times Y\times P$.
If there exist a neighborhood $\mathcal{T}\times\mathcal{U}$ of $(\bar{x},\bar{p})$ and $\theta,\tau>0$ such that 

\begin{equation}\label{semi-metric}
d(x,{\bf{S}}_{(F+G)}(p))\leq \tau d(0,F(x)+G(x,p)\cap B(\bar{k},\theta))\quad \mbox{for all}\quad (x,p)\in \mathcal{T}\times \mathcal{U},
\end{equation}
 and  $(F,G)$ is locally sum-stable around $(\bar{x},\bar{p},-\bar{k}, \bar{k}),$ then  $\bf{S}_{(F+G)}$ is Robinson metrically regular  around $(\bar{x},\bar{p})$ with modulus $\tau$.\\
 
 The conclusion remains true if the assumption of local sum stability  around $(\bar{x},\bar{p},-\bar{k}, \bar{k})$  is replaced by the following one:  $G(\bar{x},\bar{p})=\{\bar{z}\}$ and $G$ is upper semicontinuous at $(\bar{x},\bar{p}).$ 
\end{proposition}
\begin{proof}
 The proof of  this proposition is very similar to that of Proposition \ref{me-sum}. Here, we sketch the proof.
Suppose that (\ref{semi-metric}) holds for every $ (x,p)\in \mathcal{T}\times \mathcal{U}.$ Here, we can assume that $\mathcal{T}=B(\bar{x},\delta)$, with some positive $\delta>0.$
 
Since $(F,G)$ is locally sum-stable around $(\bar{x},\bar{p},-\bar{k}, \bar{k}),$  there exists $\delta>0$ such that,  for every $(x,p)\in B(\bar{x},\delta)\times \mathcal{U}$  and every $w\in (F+G)(x)\cap  B(0,\delta),$ there are $y\in F(x)\cap  B(-\bar{k},\theta)$ and $z\in G(x)\cap  B(\bar{k},\theta)$ such that $w=y+z.$

 Fix $(x,p)\in B(\bar{x},\delta) \times  \mathcal{U}.$ 
 We consider two following cases:
 
 Case 1. $d(0,F(x)+G(x,p))<\delta/2.$ Fix $\gamma>0,$ small enough so that $d(0,F(x)+G(x,p))+\gamma<\delta/2,$ and take $t\in F(x)+G(x,p)$ such that 
 $$\Vert t\Vert<d(0,F(x)+G(x,p))+\gamma.$$ Hence we have 
 $\Vert t\Vert<\delta/2,$ i.e., $t\in B(0,\delta/2)\subset B(0,\delta).$  It follows that $t\in [F(x)+G(x,p)]\cap B(0,\delta).$ \\Therefore, there are $y\in F(x)\cap  B(-\bar{k},\theta)$ and $z\in G(x,p)\cap  B(\bar{k},\theta)$ such that $t=y+z.$\\
  Consequently,  $$t\in F(x)\cap  B(-\bar{k},\theta)+G(x,p)\cap  B(\bar{k},\theta)\subset F(x)+G(x,p)\cap  B(\bar{k},\theta).$$
  It follows that $$d(0,F(x)+G(x,p)\cap  B(\bar{k},\theta) )\le \Vert t\Vert.$$
  This yields $$d(0,F(x)+G(x,p)\cap  B(\bar{k},\theta))<d(0,F(x)+G(x,p))+\gamma,$$
and therefore, as $\gamma>0$ is arbitrarily small, we derive that 
   $$d(0,F(x)+G(x,p)\cap  B(\bar{k},\theta) )\le d(0,F(x)+G(x,p)).$$
  By (\ref{semi-metric}), one derives $$d(x,{\bf{S}}_{(F+G)}(p))\leq \tau d(0,F(x)+G(x,p)),\quad \mbox{for all}\quad (x,p)\in B(\bar{x},\delta) \times  \mathcal{U}.$$
  %%%%%%%%%%%%%%%%%%%%%%%%%%%%%%%%%%%%%%%%%%%%
  Case 2. $d(0,F(x)+G(x,p))\geq \delta/2.$ According to  condition  (c), the multifunction $p\rightrightarrows G(\bar{x},\cdot)$ is lower semicontinuous at $\bar{p}. $ It follows that  the distance function $d(0, F(\bar{x})+G(\bar{x},\cdot))$ is upper semicontinuous at $\bar p,$ and thus,  there exists a neighborhood $W$ of $\bar p$ such that $$d(0, F(\bar{x})+G(\bar{x},p)\le \delta/4,\; \text{for all}\; p\in W.$$
Shrinking $W$ smaller if necessary, we can assume that $W\subset \mathcal{U}.$
Choosing $0<\delta_{1}<\min\{\delta,\tau\delta/4\}.$  For every $(x,p)\in B(\bar{x},\delta_{1}) \times  W,$  and for every small $\varepsilon>0$, there exists $u\in {\bf{S}}_{(F+G)}(p)$ such that 
  $$d(\bar{x},u)\le (1+\varepsilon)\tau d(0,F(\bar{x})+G(\bar{x},p)).$$ 
  So,
 \begin{align*}
  d(x,u)\le d(x,\bar{x})+d(\bar{x},u)\\
  &<\delta_{1}+\tau(1+\varepsilon)d(0,F(\bar x)+G(\bar x,p))\\
  &<\tau\delta/4+\tau(1+\varepsilon)\delta/4\\
  &\le \tau/2 d(0,F(x)+G(x,p))\\
  &+\tau/2(1+\varepsilon) d(0,F(x)+G(x,p).\\
    \end{align*}
Taking the limit as $\varepsilon>0$ goes to $0$, it follows that $$d(x,{\bf{S}}_{(F+G)}(p))\leq \tau d(0,F(x)+G(x,p)),$$
  establishing the proof.\qed 
  \end{proof}
The following theorem establishes the Lipschitz property for the solution mapping $\S_{\mathcal{E}_{(F,G)}}.$
 %THEOREM
\begin{theorem} \label{lipmere} Let $X$ be a complete metric space, $Y$ be a Banach space, $P$ be a topological  space. Suppose that $F: X\rightrightarrows Y$
 and $G:X\times P\rightrightarrows Y$ are  multifunctions satisfying  conditions (a), (b), (c) in Lemma \ref{(F,G)}.

If  $F$ is metrically regular around $(\bar{x},-\bar{k})$  with modulus $\tau>0$ 
and $G$ is  pseudo-Lipschitz  around $(\bar{x},\bar{p},\bar{k})$ with respect to $x$, uniformly in $p$ with modulus $\lambda>0$ such that $\tau\lambda<1,$ then $\mathcal{E}_{(F,G)}$ is metrically regular around $(\bar{x},\bar{p},\bar{k},0)$ with respect to $(x,k)$, uniformly in $p$, with modulus $(\tau^{-1}-\lambda)^{-1}.$ 

Moreover,  assume in addition that $P$ be a metric space.  If $G$ is  pseudo-Lipschitz  around $(\bar{x},\bar{p},\bar{k})$ with respect to $p$, uniformly in $x$ with modulus $\gamma>0,$ then $\S_{\mathcal{E}_{(F,G)}}$ is pseudo-Lipschitz  around $((0,\bar{p}),(\bar{x},\bar{k}))$ with modulus $L=\gamma+(\gamma+1)(\tau^{-1}-\lambda)^{-1}).$ In particular,  ${\bf{S}}_{\mathcal{E}_{(F,G)}}$ is pseudo-Lipschitz  around $((0,\bar{p}),(\bar{x},\bar{k}))$ with modulus $\gamma(1+(\tau^{-1}-\lambda)^{-1}).$
\end{theorem}
\begin{proof}
 The first part is the parametrized  version of  Theorem \ref{Slope O}. Its proof is completely similar to the one of  Theorem \ref{Slope O}, and is omitted. For the second part, as $\mathcal{E}_{(F,G)}$ is metrically regular around $(\bar{x},\bar{p},\bar{k},0)$ with respect to $(x,k)$, uniformly in $p$, with modulus $(\tau^{-1}-\lambda)^{-1},$ there exists $\delta_1>0$ such that 
\begin{equation}\label{Ine@@}
d((x,k), \S_{\mathcal{E}_{(F,G)}}(y,p))\le (\tau^{-1}-\lambda)^{-1}\varphi_{p,\mathcal{E}}((x,k),y),
\end{equation}
for all $(x,p,k,y)\in B((\bar{x},\bar{p},\bar{k},0),\delta_1).$ \\
Now, if $G$ is  pseudo-Lipschitz  around $(\bar{x},\bar{p},\bar{k})$ with respect to $p$, uniformly in $x$ with modulus $\gamma>0$ then there is $\delta_2>0$ 
 such that 
 \begin{equation}\label{Lip}
 G(x,p)\cap B(\bar{k},\delta_2)\subset G(x,p')+\gamma d(p,p') \bar{B}_Y, 
 \end{equation}
$ \text{for all}\; p,p'\in B(\bar{p},\delta_2), \text{for all}\; x\in B(\bar{x},\delta_2).$\\
Set $\alpha:=\min\{\delta_1/(\gamma+1),\delta_2\}.$ 
  Fix $(y,p),(y',p')\in B(0,\alpha)\times B(\bar{p},\alpha).$ Take $(x,k) \in \S_{\mathcal{E}_{(F,G)}}(y,p))\cap [B(\bar{x},\alpha)\times B(\bar{k},\alpha)].$\\
  %Moreover, \tcg{for any  $z\in F(x)+k,$ one has that} $$d(y',F(x)+k))\le \Vert y'-z\Vert.$$
Since $(x,k) \in \S_{\mathcal{E}_{(F,G)}}(y,p))\cap [B(\bar{x},\alpha)\times B(\bar{k},\alpha)],$ then $$ y\in F(x)+k, k\in G(x,p) \; \text{and}\; (x,k)\in B(\bar{x},\alpha)\times B(\bar{k},\alpha).$$
Along with (\ref{Lip}), we can find that $ k^\prime\in G(x,p')$ such that
$$ \|k-k^\prime\| \le \gamma d(p,p^\prime)<\gamma\alpha,$$
which follows that $k^\prime \in B(\bar k,\delta_1).$
Therefore, by (\ref{Ine@@}), one has
\begin{align*}
   d((x,k^\prime), \S_{\mathcal{E}_{(F,G)}}(y',p'))
  &\le  (\tau^{-1}-\lambda)^{-1}\varphi_{p',\mathcal{E}}((x,k^\prime),y'),\\
  &\le (\tau^{-1}-\lambda)^{-1}d(y',F(x)+k^\prime)),\\
     \end{align*}
 Hence, by noting that $y\in F(x)+k,$ one deduces that
\begin{equation}\label{@@bis}
  \end{equation}
  $$
 \begin{array}{ll}
   d((x,k), \S_{\mathcal{E}_{(F,G)}}(y',p'))
  &\le \|k-k^\prime\|+d((x,k^\prime), \S_{\mathcal{E}_{(F,G)}}(y',p'))\notag\\
&\le\gamma d(p,p^\prime)+(\tau^{-1}-\lambda)^{-1}d(y',F(x)+k^\prime)),\notag\\
  &\le \gamma d(p,p^\prime)+(\tau^{-1}-\lambda)^{-1}(\|y-y^\prime\|+\|k-k^\prime\|)\notag\\
  &\le \gamma(1+(\tau^{-1}-\lambda)^{-1})d(p,p^\prime)+(\tau^{-1}-\lambda)^{-1}\|y-y^\prime\|\notag
  \end{array}
$$
  and so
  \begin{align*}
  &\S_{\mathcal{E}_{(F,G)}}(y,p))\cap [B(\bar{x},\alpha)\times B(\bar{k},\alpha)]\\
 & \subseteq \S_{\mathcal{E}_{(F,G)}}(y',p')+Ld((y',p'),(y,p)) \bar{B}_{X}\times \bar{B}_{Y},\\
  \end{align*}
where, $L=\gamma+(\gamma+1)(\tau^{-1}-\lambda)^{-1},$ and by taking $y=y^\prime=0$ in relation (\ref{@@bis}), one also derives that ${\bf{S}}_{\mathcal{E}_{(F,G)}}$ is pseudo-Lipschitz  around $((0,\bar{p}),(\bar{x},\bar{k}))$ with modulus $\gamma(1+(\tau^{-1}-\lambda)^{-1}).$\\
  The   proof is complete. \qed 
  \end{proof}

If we add the assumption  that $(F,G)$ is locally sum-stable, we obtain the Lipschitz property of  ${\bf S}_{(F+G)}$. %THEOREM
\begin{theorem}\label{metlip}
Let $X$ be a complete metric space and $Y$ be a Banach space, $P$ be a metric space. Suppose that $F: X\rightrightarrows Y$ and $G:X\times P\rightrightarrows Y$ satisfy   conditions (a), (b), (c) in Lemma \ref{(F,G)}. Moreover, assume that
 \item{(i)} $(F,G)$ is locally sum-stable around $(\bar{x},\bar{p},-\bar{k}, \bar{k})$;
\item{(ii)} $F$ is   metrically regular around $(\bar{x},-\bar{k})$ with modulus $\tau>0;$ 
\item{(iii)}  $G$ is  pseudo-Lipschitz  around $(\bar{x},\bar{p},\bar{k})$ with respect to $x$, uniformly in $p$ with modulus $\lambda>0$ such that $\tau\lambda<1;$ 
\item{(iv)}  $G$ is  pseudo-Lipschitz  around $(\bar{x},\bar{p},\bar{k})$ with respect to $p$, uniformly in $x$ with modulus $\gamma>0$.
Then ${\bf S}_{(F+G)}$ is  Robinson  metrically regular around $(\bar{x},\bar{p})$ with modulus $(\tau^{-1}-\lambda)^{-1}.$ Moreover,
${\bf{S}}_{(F+G)}$ is pseudo-Lipschitz  around $(\bar{x},\bar{p})$ with constant $\gamma(\tau^{-1}-\lambda)^{-1}.$
\end{theorem}
\begin{proof}
Applying Proposition \ref{lipmere}, Proposition 23 and Proposition 20, respectively, we obtain that $S_{(F+G)}$ is  Robinson metrically regular  around $(\bar{x},\bar{p})$ with modulus $(\tau^{-1}-\lambda)^{-1}.$ Thus, there exists $\delta_1>0$ such that $$d(x,{\bf S}_{(F+G)}(p))\le (\tau^{-1}-\lambda)^{-1} d(0,F(x)+G(x,p)),\;\text{for all} \;(x,p)\in B((\bar{x},\bar{p}),\delta_1).$$
On the other hand, since $G$ is  pseudo-Lipschitz  around $(\bar{x},\bar{p},\bar{k})$ with respect to $p$, uniformly in $x$ with modulus $\gamma>0,$  we can find $\delta_2>0$
 such that 
 $$ G(x,p)\cap B(\bar{k},\delta_2)\subset G(x,p')+\gamma d(p,p') \bar{B}_Y,$$
 $ \text{for all}\; p,p'\in B(\bar{p},\delta_2), \text{for all}\; x\in B(\bar{x},\delta_2).$ 
Moreover, since the pair  $(F,G)$ is locally sum-stable around $(\bar{x},\bar{p},-\bar{k}, \bar{k})$, there is $\delta_3>0$ such that,  for every $(x,p)\in B(\bar x,\delta_3)\times  B(\bar p,\delta_3)$ and every $w\in [F(x)+G(x,p)]\cap B(0,\delta_3)$, there are $y\in F(x)\cap B(-\bar{k},\delta_2), z\in G(x,p)\cap B(\bar{k},\delta_2)$ such that $w=y+z.$
 Set  $\alpha:=\min\{\delta_1,\delta_2,\delta_3\}.$ Take  $p,p'\in B(\bar{p},\alpha),$ and $x\in {\bf S}_{(F+G)}(p)\cap B(\bar{x},\alpha),$ i.e., $0\in F(x)+G(x,p)$ and $x\in  B(\bar{x},\alpha)$.\\
 Moreover, we observe that for every $w\in [F(x)+G(x,p)]\cap B(0,\alpha),$ 
  \begin{align*}
 w\in  F(x)\cap B(-\bar{k},\delta_2))+G(x,p)\cap B(\bar{k},\delta_2)\subseteq F(x)+G(x,p')+\gamma d(p,p') \bar{B}_Y.
  \end{align*}
  Thus, 
  $$ [F(x)+G(x,p)]\cap B(0,\alpha)\subseteq F(x)+G(x,p')+\gamma d(p,p') \bar{B}_Y.$$
  Since $0\in F(x)+G(x,p)$, and also $0\in  [F(x)+G(x,p)]\cap B(0,\alpha)$, thus $$0\in  F(x)+G(x,p')+\gamma d(p,p') \bar{B}_Y.$$ It follows that there is $w\in F(x)+G(x,p')$ such that   $\Vert w\Vert\le \gamma d(p,p').$
 Therefore,  $$ d(x,{\bf S}_{(F+G)}(p'))\le (\tau^{-1}-\lambda)^{-1} d(0,F(x)+G(x,p'))\le (\tau^{-1}-\lambda)^{-1} \Vert w\Vert\le \gamma(\tau^{-1}-\lambda)^{-1} d(p,p').$$
So, $${\bf S}_{(F+G)}(p)\cap B(\bar{x},\alpha)\subseteq {\bf S}_{(F+G)}(p')+\gamma(\tau^{-1}-\lambda)^{-1} d(p,p') \bar{B}_{X},$$ establishing  the proof. 
\end{proof}
 
\section{Concluding Remarks}\label{concl}

We conclude the  paper with some comments and perspectives on   metric regularity/pseudo-Lipschitzness of set-valued mappings and on the   study of the associated  variational systems.
%Note that the sum mapping that  we study here  is the most general situation we can consider because one cannot obtain good 
It is not possible to obtain effective  results on the Lipschitzness of the sum  when the  both   multifunctions $F$ and $G$   depend on the parameter $p$ (see  \cite{ABor}, and \cite {Dustru}). Similarly to \cite {Dustru}, we also used variational techniques to obtain the desired variational properties of the  sum or  to  the correspondent variational systems; however, in this article, we used the theory of error bound systematically to study metric regularity of a type of epigraphical multifunction associated to two given set-valued mappings. On one hand,  this approach, avoids the closedness of the sum mapping $F+G$,  on the  other hand,  it provides  a way to derive   variational properties of the system  associated to the epigraphical mapping without using the sum-stable property (Theorem \ref{lipmere}).
This method, allows  to study more general kinds of multifunctions,     such as composition of two set-valued mappings,  as well as variational systems associated to them. 

Moreover, we also note that if a set-valued mapping $F:X\rightrightarrows Y $ is pseudo-Lipschitz  around $(\bar x,\bar y)\in \gph F$,  then it is lower semicontinuous at $\bar x$. So, in any results above,  if we impose the  assumption of pseudo-Lipschitzness to $F$,  then  the assumption of lower semicontinuity is automatically satisfied.

\end{document}